\documentclass{amsart}
\usepackage{amssymb}
\usepackage{enumerate}
\usepackage{amsthm}
\usepackage{amsmath} 
\usepackage{pifont}
\usepackage[usenames]{color}
\usepackage{bibentry}
\usepackage{amsfonts}
\usepackage[all]{xypic}
\usepackage{graphicx}
\usepackage{psfrag}
\usepackage{verbatim}
\usepackage{pstricks}
\usepackage[latin1]{inputenc}
\usepackage[all,knot]{xy}

\title{Gordian adjacency for torus knots}
\author{Peter Feller}
\address{Universit\"at Bern, Sidlerstrasse 5, CH-3012 Bern, Switzerland}
\email{peter.feller@alumni.unibe.ch}
\urladdr{}
\thanks{The author gratefully acknowledges support by the Swiss National Science Foundation Grant 137548.}
\subjclass[2010]{57M25,  57M27, 14B07}
 \keywords{{Gordian distance}, {unknotting number}, {torus knots}, {plane curve singularities}, {adjacency}}

\newtheorem{theorem}{Theorem}
\newtheorem{corollary}[theorem]{Corollary}
\newtheorem{lemma}[theorem]{Lemma}
\newtheorem{example}[theorem]{Example}
\newtheorem{definition}[theorem]{Definition}
\newtheorem{remark}[theorem]{Remark}
\newtheorem{prop}[theorem]{Proposition}
\newtheorem{claim}[theorem]{Claim}

\newenvironment{Example}{\begin{example}\rm}{\end{example}}
\newenvironment{Definition}{\begin{definition}\rm}{\end{definition}}
\newenvironment{Remark}{\begin{remark}\rm}{\end{remark}}
\newenvironment{Claim}{\begin{claim}\rm}{\end{claim}}

\def\et{\quad\mbox{and}\quad}

\def\s{{\sigma}}
\def\so{{\sigma_\omega}}

\def\N{\mathbb{N}}

\def\R{\mathbb{R}}
\def\Z{\mathbb{Z}}
\def\C{\mathbb{C}}

\def\epsilon{\varepsilon}

\begin{document}

\begin{abstract}    
 A knot $K_1$ is called \emph{Gordian adjacent} to a knot $K_2$ if there exists a minimal
unknotting sequence for $K_2$ containing $K_1$.
We provide a sufficient condition for Gordian adjacency of torus knots via the study of knots in the thickened torus $S^1\times S^1\times\R$.
We also completely describe Gordian adjacency for torus knots of braid index 2 and 3 using Levine-Tristram signatures as obstructions to Gordian adjacency.
Our study of Gordian adjacency is motivated by the concept of adjacency for plane curve singularities. In the last section we compare these two notions of adjacency.
\end{abstract}

\maketitle

\section{Introduction}

Let $K_1$ and $K_2$ be smooth knots in $\R^3$ or $S^3$.
Their \emph{Gordian distance} $d_G(K_1,K_2)$ is the minimal number of crossing changes needed to get from $K_1$ to $K_2$,
see e.g.\ Murakami \cite{Murakami_MetricsonKnots}.
The \emph{unknotting number} $u(K)$ of a knot $K$,
which was already studied by Wendt \cite{wendt},
is the distance $d_G(K,O)$, where $O$ denotes the unknot.
The Gordian distance induces a metric on the set of (isotopy classes of) all smooth knots.
This discrete metric space is huge.
For example, every $\Z^n$ can be quasi-isometrically embedded into the subspace consisting of all torus knots
by a result of Gambaudo and Ghys \cite{GambaudoGhys_BraidsSignatures}.
In this paper we study the subspace of torus knots and the simple question,
`When is the triangle inequality $d_G(K_1,K_2)\geq d_G(K_2,O)-d_G(K_1,O)$ an equality?'

\begin{Definition}
 Let $K_1$ and $K_2$ be knots. We say $K_1$ 
is \emph{Gordian adjacent} to $K_2$, denoted by $K_1\leq_G K_2$, if $d_G(K_1,K_2)=u(K_2)-u(K_1)$.
\end{Definition}

Equivalently, a knot $K_1$ is Gordian adjacent to $K_2$
if $K_2$ can be unknotted via $K_1$,
that is, if there exists a minimal unknotting sequence for $K_2$ that contains $K_1$.
A \emph{minimal unknotting sequence} for a knot $K$ is a sequence of $u(K)+1$ knots
starting with $K$ and ending with the unknot $O$
such that any two consecutive knots are related by a crossing change, see Baader~\cite{Baader_UnknottingTK}.
The name `Gordian adjacency' is motivated by the connection to algebraic adjacency, see below.
Gordian adjacency is a partial order.

For two coprime positive integers $n\geq2$ and $m\geq2$,
we denote by $T(n,m)=T(m,n)$ the (positive) \emph{torus knot}
obtained as the standard closure of the $n$-strand positive braid $(\sigma_1\cdots\sigma_{n-1})^m$
or alternatively as the knot of the singularity $x^n-y^m$, see Section~\ref{sec:algebraically}.
The \emph{braid index} of a torus knot $T(n,m)$ is the minimum of $n$ and $m$.

Our main results on Gordian adjacency for torus knots are the following.

\begin{theorem}\label{thm:T(n,m)<T(a,b)ifn<am<b}
Let $(n,m)$ and $(a,b)$ be pairs of coprime positive integers with $n\leq a$ and $m\leq b$.
Then the torus knot $T(n,m)$ is Gordian adjacent to the torus knot $T(a,b)$.
\end{theorem}

\begin{theorem}\label{thm:T(2,n)<T(3,m)}
 Let $n$ and $m$ be positive integers with $n$ odd and $m$ not a multiple of 3. Then the torus knot $T(2,n)$ is Gordian adjacent to $T(3,m)$ if and only if
$n \leq \frac{4}{3}m+\frac{1}{3}.$
\end{theorem}

The core of the proof of Theorem~\ref{thm:T(n,m)<T(a,b)ifn<am<b} 
is a generalization to knots in $S^1\times S^1\times \R$ of the following elementary fact.
If a knot $K$ in $\R^3$ has a knot diagram with $n$ crossings, then $u(K)\leq\frac{n-1}{2}$.
The proof of Theorem~\ref{thm:T(2,n)<T(3,m)} relies on explicit constructions of the required adjacencies and on Levine-Tristram signatures as obstructions to Gordian adjacency.

As a consequence of Theorem~\ref{thm:T(n,m)<T(a,b)ifn<am<b} Gordian adjacency and Gordian distance for torus knots of a fixed braid index are completely described,
i.e.\ if a positive integer $a$ is fixed, then
\[T(a,b)\leq_GT(a,c)\text{ if and only if }b\leq c\]
for all $b,c$ coprime to $a$.
Hence,
\[d_G(T(a,b),T(a,c))=|u(T(a,b))-u(T(a,c))|=\frac{(a-1)|b-c|}{2},\]
where the second equation follows from the Milnor conjecture,
which determines the unknotting number of torus knots, see equation \eqref{equnknottinNrTorusKnot}.
For torus knots $T(a,b)$ and $T(c,d)$ of different braid indices, it is in general not clear how Gordian adjacency is characterized in terms of $a,b,c,$ and $d$.
Theorem~\ref{thm:T(2,n)<T(3,m)} provides such a characterization for the case of braid index $2$ and $3$. 

\begin{Remark}\label{rem:T(3,m)not<T(2,n)}
To completely determine Gordian adjacency for torus knots of braid index $2$ and $3$, additionally to Theorem~\ref{thm:T(2,n)<T(3,m)}, one has to show that no torus knot of braid index 3 is adjacent to a torus knot of braid index 2.
More generally, Borodzik and Livingston show that a torus knot cannot be Gordian adjacent to a torus knot of strictly smaller braid index~\cite{BorodzikLivingston_13Ar}.
For this, they use a semicontinuity property that they prove using the Heegaard Floer correction term $d$---a Spin$^c$-3-manifold invariant which was defined by
Ozsv{\'a}th and Szab{\'o}~\cite{OzsvathSzabo_03_AbsolutlyGradedFloerHomologies}. 
Using signature obstructions one can only partially prove this result, see Section~\ref{sec:LTS}.
\end{Remark}

An obvious motivation for finding Gordian adjacencies is that,
by definition, every Gordian adjacency determines the Gordian distance of the involved knots.
However, Gordian adjacencies can also lead to good estimates of Gordian distances between non-adjacent torus knots.
For example, 
the adjacencies $T(2,7)\leq_GT(2,9)$ and $T(2,7)\leq_GT(3,5)$ 
yield
\begin{eqnarray*} d_G(T(2,9),T(3,5))  & \leq & u(T(2,9))-u(T(2,7))+u(T(3,5))-u(T(2,7))\\ &=&  4-3+4-3 \quad =\quad 2.\end{eqnarray*}
The converse inequality can be proven using signatures; thus, $d_G(T(2,9),T(3,5))=2$.
Trying to generalize this example for any two torus knots $T_1$ and $T_2$
we look for the highest unknotting number $u(K)$ realized by a knot $K$, adjacent to both $T_1$ and $T_2$,
and ask if $u(T_1)-u(K)+u(T_2)-u(K)$ is close to the Gordian distance $d_G(T_1,T_2)$.
An ambitious future goal is to use such Gordian adjacencies to determine Gordian distances between all torus knots up to a constant factor,
similarly to what was done for cobordism distance by Baader~\cite{Baader_ScissorEq}.

The \emph{cobordism distance} between two knots $K_1$ and $K_2$ is defined to be
the minimal genus of a connected, oriented, and smoothly embedded surface $F$ in $S^3\times[0,1]$
with $\partial F=K_1\times\{0\}\cup K_2\times\{1\}$.
Similar to the unknotting number for the Gordian distance,
the \emph{slice genus} or \emph{4-ball genus} of a knot, denoted by $g_s$,
is the cobordism distance to the unknot $O$.
As a crossing change can be realized by a cobordism of genus 1,
the Gordian distance is larger than the cobordism distance
and a Gordian adjacency between knots $K_1$ and $K_2$ yields a cobordism of genus $u(K_2)-u(K_1)$.

Another motivation for the study of Gordian adjacency
comes from the notion of adjacency for singularities of algebraic curves in $\C^2$ studied by Arnol'd \cite{Arnold_normalforms},
which yields a notion of adjacency for algebraic knots, see Section~\ref{sec:algebraically}.
Such an adjacency of algebraic knots $K_1$ and $K_2$ yields a smooth algebraic curve $F$ in $\C^2$
such that $K_1$ and $K_2$ are realized as transversal intersection of $F$ with two spheres around the origin of different radii $r_1<r_2$, i.e.\
\[K_i=F\cap\{(x,y)\in\C^2\;\vert\;\Vert x \Vert^2+\Vert y\Vert^2=r_i^2\}\subset\{(x,y)\in\C^2\;\vert\;\Vert x \Vert^2+\Vert y\Vert^2=r_i^2\}\cong S^3.\]
By a theorem of Kronheimer and Mrowka~\cite[Corollary 1.3]{KronheimerMrowka_Gaugetheoryforemb}, known as the Thom conjecture,
the slice genus $g_s(K_i)$ of $K_i$ equals the genus of the intersection of $F$ with the ball centered at the origin of $\C^i$ of radius $r_i$;
thus, the cobordism
\[F\cap \{(x,y)\in\C^2\;\vert\;r_1^2\leq \Vert x \Vert^2+\Vert y\Vert^2\leq r_2^2\}\]
in
\[\{(x,y)\in\C^2\;\vert\;r_1^2\leq \Vert x \Vert^2+\Vert y\Vert^2\leq r_2^2\}\cong S^3\times[0,1]\]
has minimal genus $g_s(K_2)-g_s(K_1)$.
By the Milnor conjecture, a consequence of the Thom conjecture, the slice genus and the unknotting number of algebraic knots are equal,
e.g.\ for torus knots one has
\begin{equation}\label{equnknottinNrTorusKnot}
 u(T(n,m)) =  g_s(T(n,m)) = \frac{(n-1)(m-1)}{2}
\end{equation}
for all coprime positive integers $n,m$.

In summary, we know that $u$ and $g_s$ coincide on algebraic knots, and both adjacency notions,
which could be thought of as relative versions of $u$ and $g_s$, respectively, have similar properties.
For example, for both notions it holds that if $K_1$ is adjacent to $K_2$,
then $u(K_1)=g_s(K_1)\leq u(K_2)=g_s(K_2)$
and the cobordism distance equals $u(K_2)-u(K_1)=g_s(K_2)-g_s(K_1)$.
Furthermore, for both notions $T(n,m)$ is adjacent to $T(a,b)$ if $n\leq a$ and $m\leq b$,
see Theorem~\ref{thm:T(n,m)<T(a,b)ifn<am<b} and Proposition~\ref{prop:T(n,m)<T(a,b)ifn<am<b(algebraic)}.
It is then natural to wonder whether the two concepts of adjacency coincide,
for example, on torus knots.
We answer by the negative in Section~\ref{sec:algebraically},
but we give a heuristic argument supporting the conjecture that
if two torus knots are Gordian adjacent, then they are algebraically adjacent.

While algebraic adjacency comes from deformations of polynomials that have algebraic curves as zero-sets, there is a more restrictive notion---$\delta$-constant adjacency---coming from deformations of parametrizations of algebraic curves. This adjacency notion seems to be closely related to Gordian adjacency; see for example~\cite{BorodzikLivingston_13Ar}, where it is proved that $\delta$-constant adjacency of knots $K_1$ and $K_2$ implies Gordian adjacency up to certain concordances. We hope to come back to this in future work. 

To decide whether a knot is Gordian adjacent to another knot,
the unknotting numbers of the involved knots should certainly be known;
thus, even ignoring the connection to algebraic adjacency,
equality~\eqref{equnknottinNrTorusKnot} is relevant to the study of Gordian adjacency for torus knots.
It is used throughout the text.

Section~\ref{sec:T(2,?)<T(3,?)} discusses examples of Gordian adjacent torus knots of braid index 2 and 3.
In Section~\ref{sec:unknottingonsurfaces} we study unknotting of knots in $S^1\times S^1\times \R$
and use it to prove Theorem~\ref{thm:T(n,m)<T(a,b)ifn<am<b}.
Section~\ref{sec:LTS} introduces Levine-Tristram signatures as obstructions to Gordian adjacencies
and uses them to prove Theorem~\ref{thm:T(2,n)<T(3,m)}.
In Section~\ref{sec:higherIndex&Asymtotics} we study Gordian adjacencies between torus knots of higher braid indices.
The relation between algebraic and Gordian adjacency is discussed in Section~\ref{sec:algebraically}.
In particular, Proposition~\ref{prop:a,bc->b,ac} provides an infinite family of examples of algebraic adjacent torus knots
that are not Gordian adjacent.

{\bf{Acknowledgements}:} I thank Sebastian Baader for introducing me to unknotting and for his ongoing support.
Thanks also to Masaharu Ishikawa for enlightening comments and technical references that led to Proposition~\ref{prop:a,bc->b,ac} and to Maciej Borodzik for pointing me towards $\delta$-constant deformations.
Finally, I wish to thank the referee for helpful suggestions and corrections.

\section{Examples of Gordian adjacencies.}\label{sec:T(2,?)<T(3,?)}
By definition, the unknot $O$ is adjacent to every knot $K$.
Let $k$ be a positive integer.
The unknotting number of the torus knot $T(2,2k+1)$ is $k$.
A minimal unknotting sequence of $T(2,2k+1)$ is provided by
\[
T(2,2k+1)\to T(2,2k-1)\to\cdots\to T(2,5)\to T(2,3)\to O.
\]
Consequently, $T(2,2l+1)\leq_GT(2,2k+1)$ for all $l\leq k$, a simple instance of Theorem~\ref{thm:T(n,m)<T(a,b)ifn<am<b}.
We now construct explicit examples of Gordian adjacencies that are not provided by Theorem~\ref{thm:T(n,m)<T(a,b)ifn<am<b}.
Let $\lfloor\cdot\rfloor$ denote the integer part of a real number.

\begin{prop}
 \label{prop:T(2,n)<T(3,m)}
For every positive integer $k$, we have
\[
T(2,2k+1) \leq_G T(3,\lfloor\frac{3}{2}k+1\rfloor).
\]
\end{prop}

\begin{proof}
The knot $T(2,2k+1)$ is the standard closure of the braid
$\xygraph{
!{0;/r0.7pc/:}
[uuuuuu]!{\xcapv[4]@(0)} [ld]
[luu]!{\vcrossneg}
[dlll]{{k-3}\,\{}[r(3.5)u(0.5)]{\vdots}
[l(0.5)d(1.5)]!{\vcrossneg}
!{\xcapv[1]@(0)} [ld]
[rruu]!{\vcrossneg}
[l]!{\vcrossneg}
[urr]!{\xcapv[1]@(0)} [ld]
[lu]!{\xcapv[1]@(0)} [ld]
[uurr]!{\vcrossneg}
[l]!{\vcrossneg}
!{\vcrossneg}
[dlll]{{k-3}\,\{}[r(3.5)u(0.5)]{\vdots}
[r(1.5)u(2.5)]!{\xcapv[5]@(0)} [ld]
[ldd]!{\vcrossneg}
}\hspace{-40pt}~,$ where $k-3$ denotes the number of the crossings not drawn.
We introduce a 
crossing change for knots containing a part that looks (in an appropriate diagram) like the above $T(2,2k+1)$.
\begin{equation}\label{eq:crossingchange}
\xygraph{
!{0;/r0.7pc/:}
[uuuuu]{\vdots}
[d(1.4)l]!{\vcrossneg}
!{\vcrossneg}
[uurr]!{\xcapv[2]@(0)} [ld]
[l]!{\xcapv[1]@(0)}[ld]
[rruu]!{\vcrossneg}
[l]!{\vcrossneg}
[urr]!{\xcapv[1]@(0)} [ld]
[lu]!{\xcapv[1]@(0)} [ld]
[uurr]!{\vcrossneg}
[l]!{\vcrossneg}
!{\vcrossneg}
!{\vcrossneg}
[uuurr]!{\xcapv[3]@(0)} [ld]
[dd]{\vdots}
}\hspace{-1pc}=\;
\xygraph{
!{0;/r0.7pc/:}
[uuuuuu]{\vdots}
[d(1.4)l]!{\vcrossneg}
[urr]!{\xcapv[1]@(0)} [ld]
[u]!{\vcrossneg}
[ul]!{\xcapv[1]@(0)} [ld]
[ur]!{\vcrossneg}
[urr]!{\xcapv[1]@(0)} [ld]
[ul]!{\xcapv[1]@(0)} [ld]
[uurr]!{\vcrossneg}
[lu]!{\xcapv[1]@(0)} [ld]
[u(1)r(2)]!{\vcrossneg}
[uul]!{\xcapv[2]@(0)} [ld]
[r]!{\vcrossneg}
[urr]!{\xcapv[5]@(0)} [ld]
[ul]!{\xcapv[2]@(0)}[ld]
[r]!{\vcrossneg}
[urr]!{\xcapv[1]@(0)}[ld]
[ul]!{\vcrossneg}
[uuuur]!{\xcapv[2]@(0)}[dddd]
[]{\vdots}
}\hspace{-2.5pc}=\
\xygraph{
!{0;/r0.7pc/:}
[uuuuuu]{\vdots}
[d(1.4)l]!{\vcrossneg}
[urr]!{\xcapv[1]@(0)} [ld]
[u]!{\vcrossneg}
[ul]!{\xcapv[1]@(0)} [ld]
[ur]!{\vcrossneg}
[urr]!{\xcapv[1]@(0)} [ld]
[ul]!{\xcapv[1]@(0)} [ld]
[uurr]!{\vcrossneg}
[lu]!{\xcapv[1]@(0)} [ld]
[u(1)r(2)]!{\vcrossneg}
[uul]!{\xcapv[2]@(0)} [ld]
[r]!{\vcrossneg}
[urr]!{\xcapv[1]@(0)} [ld]
[u]!{\vcross}
!{\vcrossneg}
[uul]!{\xcapv[2]@(0)}[ld]
[r]!{\vcrossneg}
[urr]!{\xcapv[1]@(0)}[ld]
[ul]!{\vcrossneg}
[urr]!{\xcapv[1]@(0)}[ld]
[]{\vdots}
}\overset{\text{crossing change}}{\longleftarrow}
\xygraph{
!{0;/r0.7pc/:}
[uuuuuu]{\vdots}
[d(1.4)l]!{\vcrossneg}
[urr]!{\xcapv[1]@(0)} [ld]
[u]!{\vcrossneg}
[ul]!{\xcapv[1]@(0)} [ld]
[ur]!{\vcrossneg}
[urr]!{\xcapv[1]@(0)} [ld]
[ul]!{\xcapv[1]@(0)} [ld]
[uurr]!{\vcrossneg}
[lu]!{\xcapv[1]@(0)} [ld]
[u(1)r(2)]!{\vcrossneg}
[uul]!{\xcapv[2]@(0)} [ld]
[r]!{\vcrossneg}
[urr]!{\xcapv[1]@(0)} [ld]
[u]!{\vcrossneg}
!{\vcrossneg}
[uul]!{\xcapv[2]@(0)}[ld]
[r]!{\vcrossneg}
[urr]!{\xcapv[1]@(0)}[ld]
[ul]!{\vcrossneg}
[urr]!{\xcapv[1]@(0)}[ld]
[]{\vdots}
}=\;\;
\xygraph{
!{0;/r0.7pc/:}
[uuuuuu]{\vdots}
[d(1.4)l]!{\vcrossneg}
[urr]!{\xcapv[1]@(0)} [ld]
[u]!{\vcrossneg}
[ul]!{\xcapv[1]@(0)} [ld]
[ur]!{\vcrossneg}
[urr]!{\xcapv[1]@(0)} [ld]
[lu]!{\xcapv[1]@(0)} [ld]
[uurr]!{\vcrossneg}
[l]!{\vcrossneg}
[urr]!{\xcapv[1]@(0)} [ld]
[lu]!{\xcapv[1]@(0)}[ld]
[rruu]!{\vcrossneg}
[l]!{\vcrossneg}
[urr]!{\xcapv[1]@(0)} [ld]
[lu]!{\xcapv[1]@(0)} [ld]
[uurr]!{\vcrossneg}
[l]!{\vcrossneg}
[urr]!{\xcapv[1]@(0)} [ld]
[ul]!{\vcrossneg}
[urr]!{\xcapv[1]@(0)} [ld]
[]{\vdots}
}=\;
\xygraph{
!{0;/r0.7pc/:}
[uuuuuu]{\vdots}
[d(1.4)]!{\vcrossneg}
[ul]!{\xcapv[1]@(0)} [ld]
[ur]!{\vcrossneg}
[urr]!{\xcapv[1]@(0)} [ld]
[u]!{\vcrossneg}
[ul]!{\xcapv[1]@(0)} [ld]
[ur]!{\vcrossneg}
[urr]!{\xcapv[1]@(0)} [ld]
[u]!{\vcrossneg}
[ul]!{\xcapv[1]@(0)} [ld]
[ur]!{\vcrossneg}
[urr]!{\xcapv[1]@(0)} [ld]
[u]!{\vcrossneg}
[ul]!{\xcapv[1]@(0)} [ld]
[ur]!{\vcrossneg}
[urr]!{\xcapv[1]@(0)} [ld]
[u]!{\vcrossneg}
[ul]!{\xcapv[1]@(0)} [ld]
[ur]!{\vcrossneg}
[urr]!{\xcapv[1]@(0)} [ld]
[]{\vdots}
}
,
\end{equation}
where the first and the two last equalities are obtained 
by applying the braid relation
\[\sigma_2\sigma_1\sigma_2=\hspace{0.5pc}
\xygraph{
!{0;/r0.7pc/:}
[uuu]
[d(1.4)]!{\vcrossneg}
[ul]!{\xcapv[1]@(0)} [ld]
[ur]!{\vcrossneg}
[urr]!{\xcapv[1]@(0)} [ld]
[u]!{\vcrossneg}
[ul]!{\xcapv[1]@(0)}[dr]
}\hspace{-0.5pc}=\;
 \xygraph{
!{0;/r0.7pc/:}
[uuu]
[d(1.4)r]!{\xcapv[1]@(0)}
[ull]!{\vcrossneg}
[r]!{\vcrossneg}
[ul]!{\xcapv[1]@(0)} [ld]
[ur]!{\vcrossneg}
[urr]!{\xcapv[1]@(0)} 
.}\hspace{-0.5pc}
=\sigma_1\sigma_2\sigma_1.\]

First consider the case when $k$ is odd. We use \eqref{eq:crossingchange} inductively.
\begin{eqnarray*}
T(2,2k+1)\;\longleftarrow\hspace{-30pt}
 \xygraph{
!{0;/r0.7pc/:}
[uuuuuuuuu]!{\xcapv[4]@(0)} [ld]
[luu]!{\vcrossneg}
[dlll]{{k-5}\,\{}[r(3.5)u(0.5)]{\vdots}
[l(0.5)d(1.5)]!{\vcrossneg}
!{\xcapv[1]@(0)} [ld]
[rruu]!{\vcrossneg}
[ul]!{\xcapv[1]@(0)} [ld]
[ur]!{\vcrossneg}
[urr]!{\xcapv[1]@(0)} [ld]
[u]!{\vcrossneg}
[ul]!{\xcapv[1]@(0)} [ld]
[ur]!{\vcrossneg}
[urr]!{\xcapv[1]@(0)} [ld]
[u]!{\vcrossneg}
[ul]!{\xcapv[1]@(0)} [ld]
[ur]!{\vcrossneg}
[urr]!{\xcapv[1]@(0)} [ld]
[u]!{\vcrossneg}
[ul]!{\xcapv[1]@(0)} [ld]
[ur]!{\vcrossneg}
[urr]!{\xcapv[1]@(0)} [ld]
[u]!{\vcrossneg}
[ul]!{\xcapv[1]@(0)} [ld]
[ur]!{\vcrossneg}
[urr]!{\xcapv[1]@(0)} [ld]
[lu]!{\vcrossneg}
[dlll]{{k-5}\,\{}[r(3.5)u(0.5)]{\vdots}
[r(1.5)u(2.5)]!{\xcapv[5]@(0)} [ld]
[ldd]!{\vcrossneg}
}\hspace{-35pt}=\;
 \xygraph{
!{0;/r0.7pc/:}
[uuuuuuuuu]!{\xcapv[4]@(0)} [ld]
[luu]!{\vcrossneg}
[dlll]{{k-5}\,\{}[r(3.5)u(0.5)]{\vdots}
[l(0.5)d(1.5)]!{\vcrossneg}
!{\xcapv[1]@(0)} [ld]
[rruu]!{\vcrossneg}
[ul]!{\xcapv[1]@(0)} [ld]
[ur]!{\vcrossneg}
[urr]!{\xcapv[1]@(0)} [ld]
[u]!{\vcrossneg}
[ul]!{\xcapv[1]@(0)} [ld]
[ur]!{\vcrossneg}
[urr]!{\xcapv[1]@(0)} [ld]
[lu]!{\vcrossneg}
[dlll]{{k-5}\,\{}[r(3.5)u(0.5)]{\vdots}
[r(1.5)u(2.5)]!{\xcapv[5]@(0)} [ld]
[ldd]!{\vcrossneg}
[r]!{\vcrossneg}
[ul]!{\xcapv[1]@(0)} [ld]
[ur]!{\vcrossneg}
[urr]!{\xcapv[1]@(0)} [ld]
[u]!{\vcrossneg}
[ul]!{\xcapv[1]@(0)} [ld]
[ur]!{\vcrossneg}
[urr]!{\xcapv[1]@(0)} [ld]
[u]!{\vcrossneg}
[ul]!{\xcapv[1]@(0)} [ld]
[ur]!{\vcrossneg}
[urr]!{\xcapv[1]@(0)} [ld]
}\hspace{-35pt}\longleftarrow\hspace{-20pt}
 \xygraph{
!{0;/r0.7pc/:}
[uuuuuuuuuuuu]!{\xcapv[4]@(0)} [ld]
[luu]!{\vcrossneg}
[dlll]{{k-7}\,\{}[r(3.5)u(0.5)]{\vdots}
[l(0.5)d(1.5)]!{\vcrossneg}
!{\xcapv[1]@(0)} [ld]
[rruu]!{\vcrossneg}
[ul]!{\xcapv[1]@(0)} [ld]
[ur]!{\vcrossneg}
[urr]!{\xcapv[1]@(0)} [ld]
[u]!{\vcrossneg}
[ul]!{\xcapv[1]@(0)} [ld]
[ur]!{\vcrossneg}
[urr]!{\xcapv[1]@(0)} [ld]
[u]!{\vcrossneg}
[ul]!{\xcapv[1]@(0)} [ld]
[ur]!{\vcrossneg}
[urr]!{\xcapv[1]@(0)} [ld]
[u]!{\vcrossneg}
[ul]!{\xcapv[1]@(0)} [ld]
[ur]!{\vcrossneg}
[urr]!{\xcapv[1]@(0)} [ld]
[u]!{\vcrossneg}
[ul]!{\xcapv[1]@(0)} [ld]
[ur]!{\vcrossneg}
[urr]!{\xcapv[1]@(0)} [ld]
[lu]!{\vcrossneg}
[dlll]{{k-7}\,\{}[r(3.5)u(0.5)]{\vdots}
[r(1.5)u(2.5)]!{\xcapv[5]@(0)} [ld]
[ldd]!{\vcrossneg}
[r]!{\vcrossneg}
[ul]!{\xcapv[1]@(0)} [ld]
[ur]!{\vcrossneg}
[urr]!{\xcapv[1]@(0)} [ld]
[u]!{\vcrossneg}
[ul]!{\xcapv[1]@(0)} [ld]
[ur]!{\vcrossneg}
[urr]!{\xcapv[1]@(0)} [ld]
[u]!{\vcrossneg}
[ul]!{\xcapv[1]@(0)} [ld]
[ur]!{\vcrossneg}
[urr]!{\xcapv[1]@(0)} [ld]
}
\hspace{-35pt}=\;
 \xygraph{
!{0;/r0.7pc/:}
[uuuuuuuuuuuu]!{\xcapv[4]@(0)} [ld]
[luu]!{\vcrossneg}
[dlll]{{k-7}\,\{}[r(3.5)u(0.5)]{\vdots}
[l(0.5)d(1.5)]!{\vcrossneg}
!{\xcapv[1]@(0)} [ld]
[rruu]!{\vcrossneg}
[ul]!{\xcapv[1]@(0)} [ld]
[ur]!{\vcrossneg}
[urr]!{\xcapv[1]@(0)} [ld]
[u]!{\vcrossneg}
[ul]!{\xcapv[1]@(0)} [ld]
[ur]!{\vcrossneg}
[urr]!{\xcapv[1]@(0)} [ld]
[lu]!{\vcrossneg}
[dlll]{{k-7}\,\{}[r(3.5)u(0.5)]{\vdots}
[r(1.5)u(2.5)]!{\xcapv[5]@(0)} [ld]
[ldd]!{\vcrossneg}
[r]!{\vcrossneg}
[ul]!{\xcapv[1]@(0)} [ld]
[ur]!{\vcrossneg}
[urr]!{\xcapv[1]@(0)} [ld]
[u]!{\vcrossneg}
[ul]!{\xcapv[1]@(0)} [ld]
[ur]!{\vcrossneg}
[urr]!{\xcapv[1]@(0)} [ld]
[u]!{\vcrossneg}
[ul]!{\xcapv[1]@(0)} [ld]
[ur]!{\vcrossneg}
[urr]!{\xcapv[1]@(0)} [ld]
[u]!{\vcrossneg}
[ul]!{\xcapv[1]@(0)} [ld]
[ur]!{\vcrossneg}
[urr]!{\xcapv[1]@(0)} [ld]
[u]!{\vcrossneg}
[ul]!{\xcapv[1]@(0)} [ld]
[ur]!{\vcrossneg}
[urr]!{\xcapv[1]@(0)} [ld]
[u]!{\vcrossneg}
[ul]!{\xcapv[1]@(0)} [ld]
[ur]!{\vcrossneg}
[urr]!{\xcapv[1]@(0)} [ld]
}\hspace{-40pt}\\
\\
\\
\underbrace{\longleftarrow\;\;\cdots\;\;\longleftarrow\;\;\cdots\;\;\cdots\;\;\cdots\;\;\longleftarrow\;\;\cdots\;\;\longleftarrow}_{\frac{k-5}{2}\text{ crossing changes}}\;\;T(3,3\frac{k-1}{2}+2),
\end{eqnarray*}
where every arrow indicates a 
crossing change as in~\eqref{eq:crossingchange} and the equalities are obtained by using that the full twist
$~\xygraph{
!{0;/r0.7pc/:}
[u(3)]!{\vcrossneg}
[ul]!{\xcapv[1]@(0)}
!{\vcrossneg}
[urr]!{\xcapv[1]@(0)}
[l]!{\vcrossneg}
[ul]!{\xcapv[1]@(0)}
!{\vcrossneg}
[urr]!{\xcapv[1]@(0)}
[l]!{\vcrossneg}
[ul]!{\xcapv[1]@(0)}
!{\vcrossneg}
[urr]!{\xcapv[1]@(0)}
}\hspace{-5pt}$
commutes with every $3$-braid. Thus,
\begin{eqnarray*}
d_G\left(T(2,2k+1),T(3,3\frac{k-1}{2}+2)\right)  \leq  \frac{k-1}{2}  = (3\frac{k-1}{2}+1)-k\\
 \overset{\eqref{equnknottinNrTorusKnot}}{=} u(T(3,3\frac{k-1}{2}+2))-u(T(2,2k+1)).
\end{eqnarray*}
The case when $k$ is even has essentially the same proof except that the last crossing change does not use \eqref{eq:crossingchange} but a slight variation of it.
\end{proof}
\section{Unknotting on the torus and proof of Theorem~\ref{thm:T(n,m)<T(a,b)ifn<am<b}}\label{sec:unknottingonsurfaces}
Knots in $\R^3$ can be studied via knot diagrams on $\R^2$ up to Reidemeister equivalence.
Similarly, for a surface $F$ knots in $F\times \R$ can be studied via knot diagrams on $F$.

In a knot diagram on $\R^2$ with $n$ crossings one needs to change at most $\lfloor\frac{n-1}{2}\rfloor$ of the crossings to get the unknot.
This is easily proved geometrically by drawing a knot in $\R^3$ that projects to the curve on $\R^2$ given by the diagram and that descends (or ascends) monotonically except over one point in the diagram,
see Figure~\ref{fig:unknottingintheplane},
\begin{figure}[h]
\centering
\psfrag{p}{$p$}
\psfrag{-->}{$\longrightarrow$}
\includegraphics[width=.9\textwidth]{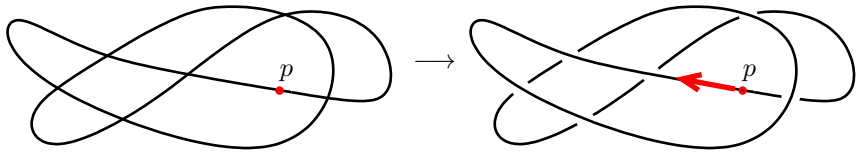}
\caption{Any curve $c$ in $\R^2$ is the projection of the unknot in $\R^3$ given by starting at any point $p$ in $\R^3$ that projects to $c$ and then descending while following $c$.}
\label{fig:unknottingintheplane}
\end{figure}
and remarking that such a knot is the unknot.
To prove Theorem~\ref{thm:T(n,m)<T(a,b)ifn<am<b}, which is a statement entirely about knots in $\R^3$, one is surprisingly led to ask whether a similar fact holds for knots in $S^1\times S^1\times \R$.
We provide such a result, which we then use to prove Theorem~\ref{thm:T(n,m)<T(a,b)ifn<am<b}.

Let $F$ be a surface. In what follows a closed smooth curve $c\colon[0,1]\to F$ is called \emph{presimple} if its lift $\tilde{c}\colon \R\to\tilde{F}$ to the universal cover $\tilde{F}$ of $F$ is injective
and if $c$ is homotopic to a simple closed curve.
A knot in $F\times\R$ that is isotopic to a knot that projects to a simple closed curve on $F$ is called \emph{unknotted}.
\begin{Remark}\label{rmk:unknot} There is at most one unknot (up to isotopy) in every homotopy class of closed curves in $F\times\R$.
This follows from the fact that homotopic simple closed curves in surfaces are isotopic, see Epstein \cite{Epstein_66_curves}.

In the case of the torus we can be more precise. 
A homotopy class of closed curves in $S^1\times S^1\times\R$ contains an unknot, which is unique up to isotopy,
if and only if (via the usual identification of $\pi_1(S^1\times S^1)\cong \pi_1(S^1\times S^1\times \R)$ with $\Z^2$) the corresponding element in $\Z^2$
has coprime entries or is $(0,0)$.
This is a reformulation of the classification of simple closed curves in $S^1\times S^1$, written, for example, in Rolfsen's textbook\ \cite{rolfsen_knotsandlinks}.\end{Remark}
\begin{lemma}\label{lemma:entknotendurchwechselnderhaelftederKreuzungen}
For every 
presimple curve $c$ in $S^1\times S^1$ 
there is a knot $O$ in $S^1\times S^1\times\R$ that projects to $c$ on $S^1\times S^1$ and that is unknotted.
\end{lemma}
\begin{Remark}\label{rmk:entknotendurchwechselnderhaelftederKreuzungen}In terms of knot diagrams Lemma~\ref{lemma:entknotendurchwechselnderhaelftederKreuzungen} means that
if a knot $K$ in $S^1\times S^1\times \R$ projects to a presimple diagram with $n$ crossings on $S^1\times S^1$,
then one can get the diagram of the unknot by changing at most $\lfloor\frac{n}{2}\rfloor$ of the $n$ crossings. 

To prove this, we use Lemma~\ref{lemma:entknotendurchwechselnderhaelftederKreuzungen} to get the unknot $O$ with the same diagram as $K$, except it differs in the choice of crossings.
If this new diagram differs from the original one in less than half of the crossings, we are done. Otherwise
we switch all 
crossings in the diagram of $O$ 
yielding a knot diagram of a knot $\overline{O}$.
The knot $\overline{O}$ is also unknotted, as the following shows.
Let $H_t$ be an isotopy 
that changes $O$ to a knot that projects to a simple closed curve on $S^1\times S^1$.
Then parametrize $\overline{O}$ in $S^1\times S^1\times\R$ exactly the same way as $O$,
except changing the sign in the $\R$ coordinate.
The same isotopy $H_t$ as for $O$ (with a change of sign in the last coordinate) shows that $\overline{O}$ 
is unknotted.
\end{Remark}

Clearly the assumption that $c$ is homotopic to a simple closed curve is necessary in Lemma~\ref{lemma:entknotendurchwechselnderhaelftederKreuzungen}.
We conjecture that
Lemma~\ref{lemma:entknotendurchwechselnderhaelftederKreuzungen} holds for all curves $c$ that are homotopic to a simple closed curve and,
furthermore, that Lemma~\ref{lemma:entknotendurchwechselnderhaelftederKreuzungen} generalizes to all surfaces.

\begin{proof}[Proof of Lemma~\ref{lemma:entknotendurchwechselnderhaelftederKreuzungen}]
Denote $S^1\times S^1$ by $F$. Our strategy is to construct a presimple homotopy $h_t$ of $c$ (meaning $h_t$ is presimple for every $t\in[0,1]$) to a simple closed curve
and then to find an isotopy $H_t$ of knots in $F\times \R$ that has $h_t$ as projection.

We first lift the curve $c$ to a mapping $\tilde{c}\colon\R\to \tilde{F}$, where $\varphi\colon\tilde{F}\to F$ denotes the universal covering map. 
Since $c$ is presimple,
$\tilde{c}\colon\R\to\tilde{F}$ is injective
and there exists a simple closed curve $g\colon[0,1]\to F$ that is homotopic to $c$. We take $g$ such that $g(0)=g(1)=c(0)=c(1)$
and denote by $\tilde{g}\colon\R\to\tilde{F}$ its lift to $\tilde{F}$ with $\tilde{g}(k)=\tilde{c}(k)$ for all $k\in\Z$.
Let $\tilde{h}_t\colon\R\to \tilde{F}$ be an equivariant\footnote{I.e.\ $\tilde{h}_t(s+1)=D(\tilde{h}_t(s))$ for all $s$ in $\R$, where $D$ denotes the unique deck transformation sending $\tilde{c}(0)$ to $\tilde{c}(1)$.}
isotopy between $\tilde{c}$ and $\tilde{g}$ that is constant on $\Z$, see Figure~\ref{fig:equivariantIsotopy}.
\begin{figure}[h]
\centering
\psfrag{c(0)=g(0)}{$\tilde{c}(0)=\tilde{g}(0)$}
\psfrag{c(1)=g(1)}{$\tilde{c}(1)=\tilde{g}(1)$}
\includegraphics[width=0.9\textwidth]{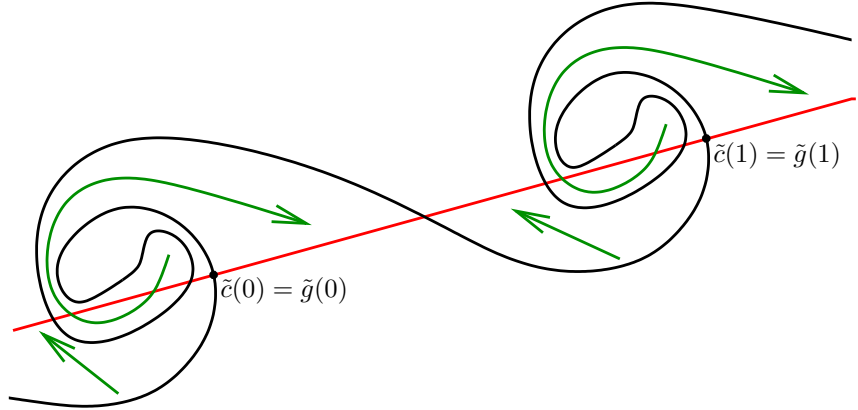}
\caption{An equivariant isotopy (green) of $\tilde{c}$ (black) to $\tilde{g}$ (red) is indicated.}
\label{fig:equivariantIsotopy}
\end{figure}
Of course $h_t=\varphi\circ\tilde{h}_t\colon[0,1]\to F$ is a presimple homotopy.

The idea for building $H_t$ is to measure how far away from $g$ points $p=h_t(s)$ are
and then to put this distance $d(p)$ in the second coordinate of $H_t$.
We need a metric to make this precise and the distance will actually be measured in the universal cover.
Put a Riemannian metric on $F$ with constant curvature 0
such that $g$ is a simple closed geodesic of length 1.
The universal cover $\tilde{F}$ is identified with the Euclidean plane $\R^2$
such that $\varphi\colon\tilde{F}\to F$ is locally an isometry.
Let $d\colon\tilde{F}\to\R$ denote the oriented distance
 to the straight line $\tilde{g}$.\footnote{Ordinary Euclidean distance of points in $\tilde{F}=\R^2$ to the straight line $\tilde{g}$
with a sign depending on whether the point is on the left or the right of $\tilde{g}$.}
We claim that the homotopy
\[
H_t\colon[0,1]\to F\times\R, s\mapsto (h_t(s),d(\tilde{h}_t(s))),
\]
which projects to the homotopy $h_t$ on $F$, is an isotopy.
This claim implies that $H_0\colon[0,1]\to F\times\R$ is an unknot $O$ that projects to $h_0=c$;
therefore, it finishes the proof.

In order to prove that $H_t$ is an isotopy,
we assume towards a contradiction that $H_t$ is not injective for some fixed $t$.
Without loss of generality we assume $t=0$, i.e.\ $\tilde{h}_t=\tilde{c}$.
If there exist $s \neq r\in[0,1)$ such that $H_0(s)=H_0(r)$, then, by definition of $H_0$, the points $\tilde{p_1}=\tilde{c}(s)$ and $\tilde{p_2}=\tilde{c}(r)$ in $\tilde{F}$ satisfy
\[\varphi(\tilde{p_1})=\varphi(\tilde{p_2})\quad\et\quad d(\tilde{p_1})=d(\tilde{p_2}).
\]
As $d(\tilde{p_1})=d(\tilde{p_2})$, there is a geodesic segment parallel to $\tilde{g}$ from $\tilde{p_1}$ to $\tilde{p_2}$.
The length of this segment is an integer $k$ since $\varphi(\tilde{p_1})=\varphi(\tilde{p_2})$.
It follows that $\tilde{p_2}=\tilde{c}(k+s)$ if the sign of $k$ is chosen correctly.
This is seen by lifting $c$ to $\tilde{F}$ such that the lift starts at $\tilde{g}(k)=\tilde{c}(k)$, see Figure~\ref{fig:Hisinjective} for a case with $k=1$.
\begin{figure}[h]
\centering
\psfrag{l}{$(l,0)$}
\psfrag{g}{\textcolor{red}{$\tilde{g}$}}
\psfrag{c}{$\tilde{c}$}
\psfrag{g(0)}{$\tilde{g}(0)$}
\psfrag{g(k)}{$\tilde{g}(k)$}
\psfrag{p1}{$\tilde{p}_1$}
\psfrag{p2}{$\tilde{p}_2$}
\psfrag{dp1}{\textcolor{OliveGreen}{$\vert d(\tilde{p}_1)\vert$}}
\psfrag{dp2}{\textcolor{OliveGreen}{$\vert d(\tilde{p}_2)\vert$}}
\psfrag{Liftofcstartingatg(k)}{\textcolor{blue}{Lift of $c$ starting at $\tilde{g}(k)$}}
\includegraphics[width=0.9\textwidth]{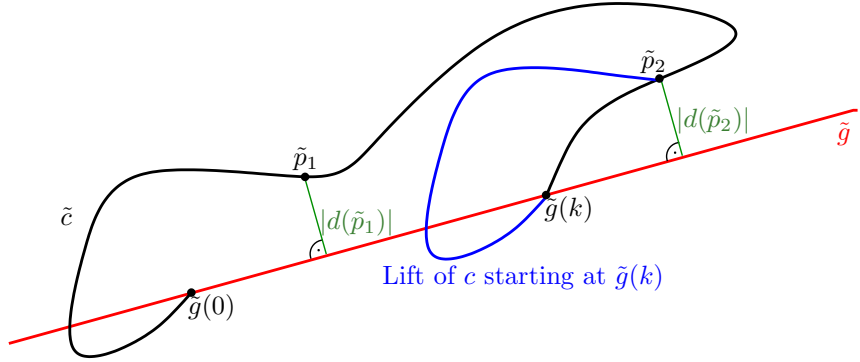}
\caption{The curve $\tilde{c}\vert_{[0,1]}$ (black) intersects $\tilde{c}\vert_{[k,k+s]}$ (blue) in $\tilde{p}_2$.}
\label{fig:Hisinjective}
\end{figure}
However, $\tilde{c}(r)=\tilde{c}(k+s)$ and $k+s \neq r$ contradict the injectivity of $\tilde{h}_t=\tilde{c}$.

\end{proof}
Let us shortly introduce notations and the general strategy for the proof of Theorem~\ref{thm:T(n,m)<T(a,b)ifn<am<b}. 
In the following $S^1\times S^1$ denotes the standard torus in $\R^3$ 
and $N(S^1\times S^1)$ a tubular neighborhood of $S^1\times S^1$.
Also, we denote the curve obtained by projecting a knot $K$ in $N(S^1\times S^1)$ to $S^1\times S^1$  by $\pi(K)$.
Such a curve $\pi(K)$ (together with crossing information) provides a knot diagram on $S^1\times S^1$ for
the knot $K$ in $N(S^1\times S^1)\cong S^1\times S^1\times\R$.

To show the adjacency $K_1\leq_G K_2$ for the knots $K_2=T(a,b)$ and $K_1=T(n,m)$,
i.e.\ to show that $d_G(K_2,K_1)$ is less than or equal to (and thus equal to) $u(K_2)-u(K_1)$, we proceed as follows.
We isotope $K_2$ and $K_1$ into $N(S^1\times S^1)$ in such a way that
\begin{enumerate}[(I)]
\item \label{eq:K1unknotted}$\pi(K_1)$ is simple closed (thus, $K_1$ is unknotted in $N(S^1\times S^1)$),
\item \label{eq:K1=K2} $K_2$ is homotopic to $K_1$ in $N(S^1\times S^1)$,
\item \label{eq:NrofcrossingsofK1} and $\pi(K_2)$ has $2(u(K_2)-u(K_1))$ crossings.
\end{enumerate}
In all our cases $\pi(K_2)$ will have an injective lift to the universal cover $\R^2$.
This together with \eqref{eq:K1unknotted} and \eqref{eq:K1=K2} 
yields that $\pi(K_2)$ is a presimple curve in $S^1\times S^1$.
Hence, Remark~\ref{rmk:entknotendurchwechselnderhaelftederKreuzungen} applies
and, because of \eqref{eq:NrofcrossingsofK1}, guaranties the existence of $u(K_2)-u(K_1)$ crossing changes in $N(S^1\times S^1)\cong S^1\times S^1\times\R$ changing $K_2$ to the unknot.
This unknot is homotopic to $K_1$ by \eqref{eq:K1=K2} and thus isotopic to $K_1$ by Remark~\ref{rmk:unknot}.

Before giving a proof of Theorem~\ref{thm:T(n,m)<T(a,b)ifn<am<b}, we illustrate this strategy in a concrete example.
\begin{Example}\label{Ex:T(3,5)<T(3,7)}
We show that $T(3,5)$ is Gordian adjacent to $T(3,7)$. 
Since $u(T(3,7))-u(T(3,5))=2$, we need to show that we can change $T(3,7)$ to $T(3,5)$ via 2 crossing changes.
First we isotope $T(3,7)$ into 
$N(S^1\times S^1)$ 
as shown on the left-hand side in Figure~\ref{fig:T(3,7)>T(3,5)}.
\begin{figure}[h]
\centering
\includegraphics[width=0.9\textwidth]{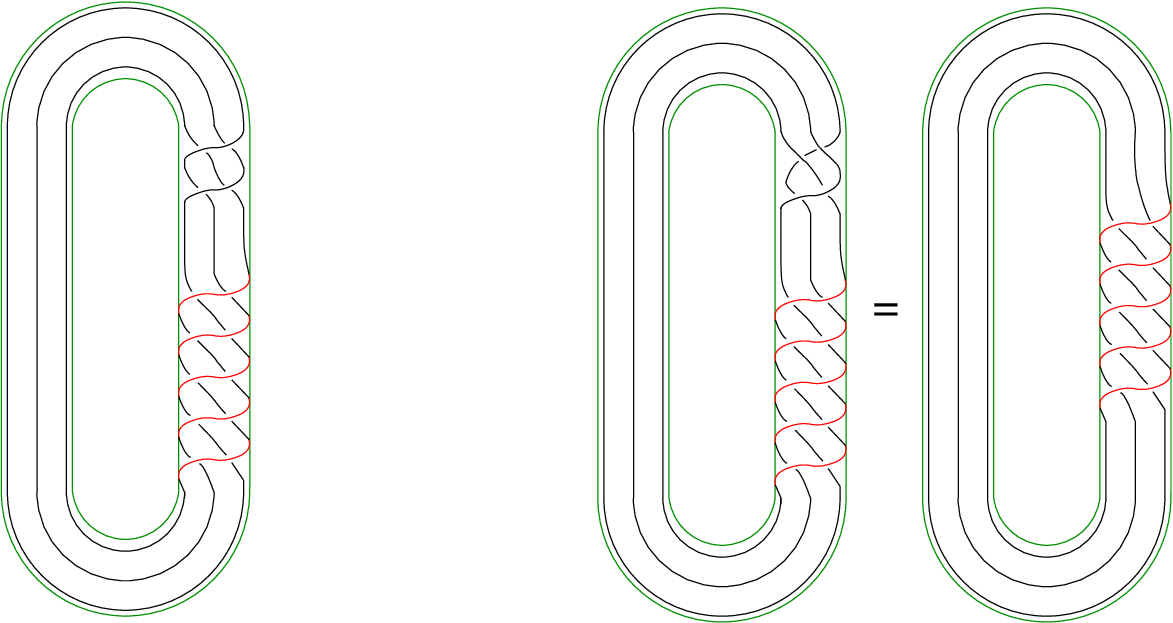}
\caption{Knots contained in a tubular neighborhood of the standard torus (green) that are homotopic in this neighborhood. Five arcs (red) are on the upper half of the torus,
the rest of the knots (black) lies on the lower half. Left: The knot $T(3,7)$ with $4$ crossings when projected on to the torus.
Right: Two isotopic (in a neighborhood of the torus) occurrences of the knot $T(3,5)$, one of them without crossings.}
\label{fig:T(3,7)>T(3,5)}
\end{figure}
Projecting this $T(3,7)$ to $S^1\times S^1$ yields a curve $\pi(T(3,7))$ with $4$ crossings.
The curve $\pi(T(3,7))$ is presimple since it has an injective lift to $\R^2$ and is homotopic to the standard embedding of the torus knot $T(3,5)$.
Thus, by Remark~\ref{rmk:entknotendurchwechselnderhaelftederKreuzungen} changing 2 of the crossings suffices to produce a knot $K$ in $N(S^1\times S^1)$
that is unknotted. As the knot $K$ and the standard $T(3,5)$ are homotopic unknots in $N(S^1\times S^1)$, they are isotopic in $N(S^1\times S^1)$ by Remark~\ref{rmk:unknot}.
In particular, $K$ and $T(3,5)$ are isotopic as knots in $\R^3$; thus, $d_G(T(3,5),T(3,7))=2$.
In this example with only 4 crossings one can quickly exhibit the knot $K$ explicitly. E.g.\ the right-hand side of Figure~\ref{fig:T(3,7)>T(3,5)} provides a knot $K$
that is obtained from the knot on the left-hand side of Figure~\ref{fig:T(3,7)>T(3,5)} by performing two crossing changes in $N(S^1\times S^1)$
and that is isotopic to the standard $T(3,5)$ as predicted by Remark~\ref{rmk:entknotendurchwechselnderhaelftederKreuzungen}. This last isotopy can be seen by applying braid relations
(similarly as in the proof of Proposition~\ref{prop:T(2,n)<T(3,m)}) and checking that these can be realized while staying within $N(S^1\times S^1)$.
\end{Example}

\begin{proof}[Proof of Theorem~\ref{thm:T(n,m)<T(a,b)ifn<am<b}]
By assumption we have pairs of coprime positive integers $(a,b)$ and $(n,m)$ such that $n\leq a$ and $m\leq b$. Without loss of generality we suppose that $a<b$ and $n<m$.

Let us first consider the case $n=a$, for which we proceed as in Example~\ref{Ex:T(3,5)<T(3,7)}.
We need to show that $d_G(T(a,b),T(n,m))$ is equal to
\[
u(T(a,b))-u(T(n,m))=\frac{(b-1)(a-1)}{2}-\frac{(m-1)(n-1)}{2}=\frac{(b-m)(a-1)}{2}.
\]
We consider the knot $T(a,b)$ as the closure of the braid $(\sigma_1\sigma_2\cdots\sigma_{a-1})^{b}$
and isotope it into a neighborhood $N(S^1\times S^1)$ of the standard torus $S^1\times S^1$ in $\R^3$. More precisely, we isotope $m$ arcs on the upper half of the torus and the rest of $T(a,b)$ on the lower half of the torus
in such a way that the curve $\pi(T(a,b))$ winds $m$ times around the core of $S^1\times S^1$ and $n=a$ times in the direction of the core of $S^1\times S^1$,
see left-hand side of Figure~\ref{fig:T(3,7)>T(3,5)}. Since $n$ and $m$ are coprime, there is a simple closed curve in $S^1\times S^1$ 
that is homotopic to $\pi(T(a,b))$ by the second part of Remark~\ref{rmk:unknot};
namely, the standard embedding of the torus knot $T(n,m)$ in $S^1\times S^1$. Also, $\pi(T(a,b))$ lifts injectively to the universal cover $\R^2$; thus, $\pi(T(a,b))$ is presimple.
The $m$ arcs do not intersect the rest of the curve $\pi(T(a,b))$ on the torus, so $\pi(T(a,b))$ has $(b-m)(a-1)$ crossings on the torus.
By Remark~\ref{rmk:entknotendurchwechselnderhaelftederKreuzungen} we need to change at most $\frac{(b-m)(a-1)}{2}$ crossings in the diagram on the torus
(which correspond to crossing changes in $N(S^1\times S^1)\cong S^1\times S^1\times\R$) to get an unknot $K$ in $N(S^1\times S^1)$. As the unknotted $K$ and the standard $T(n,m)$ are homotopic in $N(S^1\times S^1)$,
they are also isotopic by Remark~\ref{rmk:unknot}. Of course, $K$ is isotopic to $T(n,m)$ in $\R^3$
via the same isotopy as in $N(S^1\times S^1)$. 
Therefore, \[d_G(T(a,b),T(n,m))\leq \frac{(b-m)(a-1)}{2}\] as we wanted. The same argument works if $m=b$ or $a=m$.

This leaves the case $n<a$ and $m<b$. In the first case we interpreted $T(a,b)$ as the closure of a braid on $a$ strands, in the following we see $T(a,b)=T(b,a)$ as a braid on $b$ strands.
We may assume $m>b-a$, otherwise we replace (inductively) $a,b$ by $a,b-a$ (respectively by $b-a,a$ if $b-a<a$) since by the first case $T(a,b-a)\leq_GT(a,b)$.
To apply the same idea as before we reduce the braid on $b$ strands to one on $m$ strands. 
More precisely, the representation of $T(a,b)$ as the closure of the $b$-strand braid
\begin{equation}\label{eq:sigma=sigma'}
(\sigma_1\cdots\sigma_{b-1})^a=\sigma_{a}\cdots\sigma_1(\sigma_2\cdots\sigma_{b-1})^a
\end{equation}
 has the same closure as the $b-1$-strand braid
\[
\tau_{b-1}=\sigma_{a-1}\cdots\sigma_1(\sigma_1\cdots\sigma_{b-2})^a,
\]
see Figure~\ref{fig:sigmas}.
\begin{figure}[h]
\centering
\psfrag{n+1}{$a+1$}
\psfrag{n}{$a$}
\psfrag{m}{$b$}
\psfrag{m-1}{$b-1$}
\includegraphics[width=0.9\textwidth]{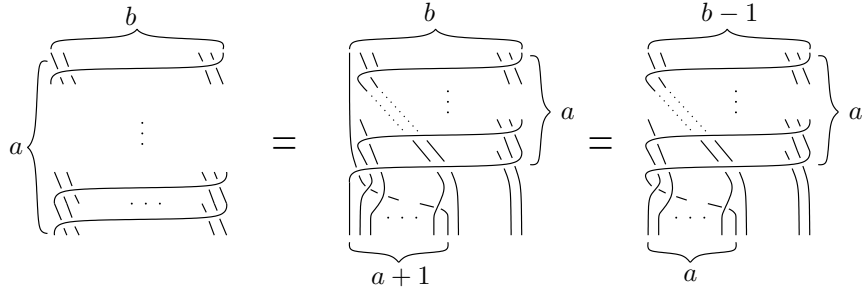}
\caption{The first equality is the pictorial version of equation \eqref{eq:sigma=sigma'}.
The second equality is a Markov destabilization, i.e.\ a Reidemeister I move on the closed braid.}
\label{fig:sigmas}
\end{figure}
If $m=b-1$, we isotope $T(a,b)$ (seen as the closure of $\tau_{b-1}$) into $N(S^1\times S^1)$ such that $n$ of the $a$ over-passing arcs in the right part of Figure~\ref{fig:sigmas} project to the upper half of the torus
and the rest of $\pi(T(a,b))$, including $a-1+(a-n)(b-2)$ crossings, lies on the lower half. The curve $\pi(T(a,b))$ is presimple since it winds $n$ respectively $m$ times around the torus,
i.e.\ it is homotopic in $N(S^1\times S^1)$ to the standard embedding of the knot $T(n,m)$, and $\pi(T(a,b))$ lifts injectively to $\R^2$.
Therefore, we can use Remark~\ref{rmk:entknotendurchwechselnderhaelftederKreuzungen} to get $T(n,m)$ by at most $\frac{a-1+(a-n)(b-2)}{2}$ crossing changes.
Thus, $d_G(T(n,m),T(a,b))$ is less or equal to
\[
\frac{a-1+(a-n)(b-2)}{2}=\frac{(a-1)(b-1)}{2}-\frac{(n-1)(b-2)}{2}=u(T(a,b))-u(T(n,m)).
\]

Now suppose $m<b-1$. We no longer isotope $T(a,b)$ into $N(S^1\times S^1)$. We first apply some crossing changes in $\R^3$ and then isotope the result into  $N(S^1\times S^1)$.
More precisely, we change a crossing in $\tau_{b-1}$ to get
\begin{equation}\label{eq:m<b-1}
\sigma_{a-1}\cdots\sigma_2\sigma_1^{-1}(\sigma_1\cdots\sigma_{b-2})^a
=\sigma_{a-1}\cdots\sigma_2\sigma_2\cdots\sigma_{b-2}(\sigma_1\cdots\sigma_{b-2})^{a-1}
\end{equation}
and then replace in \eqref{eq:m<b-1} the part $(\sigma_1\cdots\sigma_{b-2})^{a-1}$ by $\sigma_{a-1}\cdots\sigma_1(\sigma_2\cdots\sigma_{b-2})^{a-1}$ as in \eqref{eq:sigma=sigma'},
which has the same closure as the $b-2$ braid
\[
\tau_{b-2}=(\sigma_{a-2}\cdots\sigma_1\sigma_1\cdots\sigma_{b-3})^2(\sigma_1\cdots\sigma_{b-3})^{a-2},
\]
see Figure~\ref{fig:sigmas'}.
\begin{figure}[h]
\centering
\psfrag{n-1}{$a-1$}
\psfrag{->}{$\longrightarrow$}
\psfrag{n}{$a$}
\psfrag{m-1}{$b-1$}
\psfrag{m-2}{$b-2$}
\includegraphics[width=0.9\textwidth]{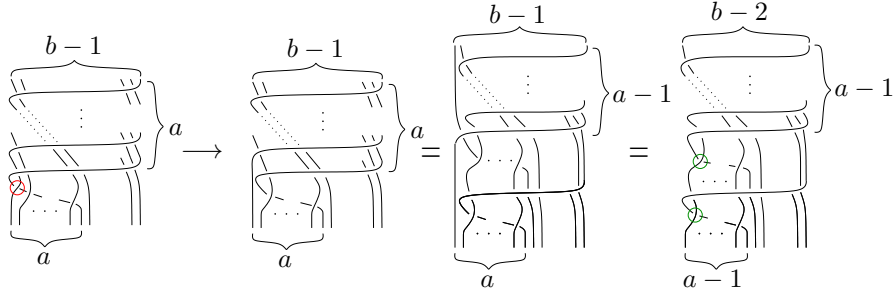}
\caption{The arrow $\longrightarrow$ indicates the changing of the marked (red) crossing. The two equalities are seen as in Figure~\ref{fig:sigmas}.
The two marked (green) crossings on the right side indicate the crossing changes that are necessary to obtain $\tau_{b-3}$ from $\tau_{b-2}$, which is needed when $m<b-2$.}
\label{fig:sigmas'}
\end{figure}
If $m=b-2$, we isotope the closure of $\tau_{b-2}$ into $N(S^1\times S^1)$ in such away that it is homotopic to $T(n,m)$;
namely, such that $n$ of the $a$ over-passing arcs get to lie on the upper part of the torus and the reminding part (including $2(a-2) + (a-n)(b-3)$ crossings) lies on the lower part.
Therefore, Remark~\ref{rmk:entknotendurchwechselnderhaelftederKreuzungen} implies that $T(n,m)$ can be obtained from the closure of $\tau_{b-2}$ by changing $\frac{ 2(a-2)+(a-n)(b-3)}{2}$ crossings.
Thus, $d_G(T(n,m),T(a,b))$ is less than or equal to \begin{align*}
 1 + \frac{ 2(a-2) + (a-n)(b-3)}{2}& = \frac{2a-2 + (a-n)(b-3)}{2}\\
&=\frac{2a-2 + (a-1)(b-3)}{2}-\frac{(n-1)(b-3)}{2}\\
&=\frac{(a-1)(b-1)}{2}-\frac{(n-1)(b-3)}{2}\\
&=u(T(a,b))-u(T(n,m)).\end{align*}

For general $m>b-a$, it follows similarly 
that we need to change \[1+2+\cdots +(b-m-1)=\frac{(b-m)(b-m-1)}{2}\] crossings of $T(a,b)$
to get the closure of the $m$ braid
\[
\tau_{m}=(\sigma_{a-(b-m)}\cdots\sigma_1\sigma_1\cdots\sigma_{m-1})^{b-m}(\sigma_1\cdots\sigma_{m-1})^{a-(b-m)},
\] see Figure~\ref{fig:sigmas''}.
\begin{figure}[h]
\centering
\psfrag{b}{$m$}
\psfrag{m-b}{$b-m$}
\psfrag{n-(m-b)}{$a-(b-m)$}
\psfrag{n-(m-b)+1}{$a-(b-m)+1$}
\includegraphics[width=0.9\textwidth]{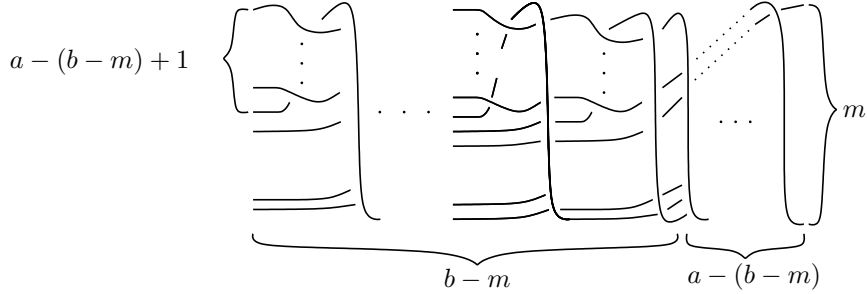}
\caption{The braid $\tau_m$, which can be obtained from $T(a,b)$ by $1+2+\cdots +(b-m-1)$ crossing changes.}
\label{fig:sigmas''}
\end{figure}
For example, the closure of $\tau_{b-3}$ is obtain from the closure of $\tau_{b-2}$ by the two crossing changes that are indicated (green) in Figure~\ref{fig:sigmas'}.
We isotope the closure of $\tau_m$ into $N(S^1\times S^1)$ such that $n$
 of the $a$ over-passing arcs lie on the upper half of the torus and $(b-m)(a-(b-m))+(a-n)(m-1)$ crossings on the lower half.
Therefore, we get $T(n,m)$ from $\tau_{m}$ by changing $\frac{(b-m)(a-(b-m))+(a-n)(m-1)}{2}$ crossings by Remark~\ref{rmk:entknotendurchwechselnderhaelftederKreuzungen}.
Combined we have that $d_G(T(n,m),T(a,b))$ is less than or equal to
\[
\frac{(b-m)(b-m-1)+(b-m)(a-(b-m))+(a-n)(m-1)}{2},
\]
which is equal to $u(T(a,b))-u(T(n,m)).$\end{proof}

\section{Levine-Tristram signatures as obstructions to adjacency}\label{sec:LTS}
The goal of this section is to prove Theorem~\ref{thm:T(2,n)<T(3,m)} using Levine-Tristram signatures \cite{levine}\cite{tristram}.
For torus knots, they are easy to calculate and yield good obstructions to adjacency,
see Lemma~\ref{lemma_signformula} and Proposition~\ref{prop:Sigundercrossingchange}, respectively.

\begin{Definition}\cite{levine}\cite{tristram}\label{Def:Sig}
Let $A$ be a Seifert matrix of a knot $K$ and $\omega$ in $S^1\backslash\{1\}\subset\C$.
The \emph{$\omega$-signature} $\sigma_\omega(K)\in\Z$ is defined to be
the number of positive eigenvalues minus the number of negative eigenvalues of the Hermitian matrix $(1-\omega)A+(1-\overline{\omega})A^t$.
\end{Definition}
The $\omega$-signature is independent of the choice of Seifert matrix, i.e.\ it is a link invariant. One has $\so=\s_{\overline{\omega}}$.
Setting $\omega=-1$ one recovers the classical signature $\sigma=\sigma_{-1}$. 

Note that there is an issue with sign conventions for the signatures (hidden in the Seifert matrix in Definition~\ref{Def:Sig}).
Our sign convention of signatures
is such that all (positive) torus knots have positive signature,
e.g.~$\sigma_{-1}(T(2,3))=2$ rather than $\sigma_{-1}(T(2,3))=-2$.

For a fixed link $L$, the signature $\so(L)$ is piecewise-constant in $\omega$, ``jumping'' at a finite number of $\omega$. 
For a Seifert matrix $A$ of a knot $K$, if $\omega$ is a root of unity of prime order, then $(1-\omega)A+(1-\overline{\omega})A^t$ is invertible, 
and so $\so(K)$ is even and $\so(K)$ does not jump at $\omega$.
From now on, every $\omega$ we consider is a root of unity of prime order. As roots of unity of prime order are dense in $S^1$, one only loses information on ``jumping''-points.



Let us denote by $s(K)$ the Rasmussen invariant of a knot $K$ \cite{rasmussen_sInv}. The next lemma
shows, how $\omega$-signatures and $s$ behave with respect to crossing changes.
\begin{lemma}
\label{lemma:Sigundercrossingchange} 
If $K_-$ is obtained from $K_+$ via one positive-to-negative crossing change
, then
\[
\sigma_\omega(K_-)\in\{\sigma_\omega(K_+),\sigma_\omega(K_+)-2\}
.\]
The same holds for the Rasmussen invariant.
\end{lemma}
Rasmussen used an observation by Livingston~\cite[Corollary 2 and 3]{Livingston_Comp} to prove Lemma~\ref{lemma:Sigundercrossingchange} for $s$ \cite{rasmussen_sInv}.
For $\omega$-signatures, we only found proofs of the following weaker statement in the literature (see~\cite[Theorem 11.2.1]{kawauchi_knottheory} and~\cite{GambaudoGhys_BraidsSignatures}).
\begin{corollary}\label{cor:K1<K2->g(K2)-g(K1)>sigmdiff/2)}
Let $K_1$, $K_2$ be knots. 
Then
\[
\vert\frac{\sigma_\omega(K_2)-\sigma_\omega(K_1)}{2}\vert\leq d_G(K_1,K_2).
\] 

In particular, if $K_1$ is adjacent to $K_2$, then $\vert\frac{\sigma_\omega(K_2)-\sigma_\omega(K_1)}{2}\vert\leq u(K_2)-u(K_1)$.
\end{corollary}
We provide a proof of Lemma~\ref{lemma:Sigundercrossingchange} at the end of this section using a variation of Livingston's observation.

As a consequence of Corollary~\ref{cor:K1<K2->g(K2)-g(K1)>sigmdiff/2)}, we prove that most torus knots are not adjacent to torus knots of braid index two, as claimed in Remark~\ref{rem:T(3,m)not<T(2,n)}.
For braid index two torus knots, the signature equals twice the unknotting number,
that is
\[
\frac{\sigma_{-1}(T(2,n))}{2}=u(T(2,n))=\frac{n-1}{2}.
\]
This is also true for $T(3,4)$ and $T(3,5)$,
but for all other torus knots $T$, there is a signature defect, i.e.\ $u(T)>\frac{\sigma_{-1}(T)}{2}$. 
Thus, by Corollary~\ref{cor:K1<K2->g(K2)-g(K1)>sigmdiff/2)},
\[
d_G(T(2,n),T)\geq \frac{\sigma_{-1}(T(2,n))}{2}-\frac{\sigma_{-1}(T)}{2}>u(T(2,n))-u(T)
\]
for all torus knots $T$ not equal to $T(3,4)$, $T(3,5)$ or some $T(2,m)$.

The following proposition explains, how Lemma~\ref{lemma:Sigundercrossingchange} gives another obstruction to Gordian adjacency of torus knots, which is often better than Corollary~\ref{cor:K1<K2->g(K2)-g(K1)>sigmdiff/2)}. 

\begin{prop}\label{prop:Sigundercrossingchange}
Let $T_1\leq_GT_2$ be a Gordian adjacency of torus knots. 
Then $\sigma_\omega(T_1)\leq\sigma_\omega(T_2)$.
\end{prop}
\begin{proof}
For all torus knots $T$, we have $\frac{s(T)}{2}=u(T)$ \cite{rasmussen_sInv}.
Thus, Lemma~\ref{lemma:Sigundercrossingchange} yields that a minimal unknotting sequence of any torus knot involves only positive-to-negative crossing changes since $s$ has to drop by $2$ with every crossing change.
Choose an $\widetilde{\omega}$ that is regular for every knot and such that $\sigma_\omega(T_1)=\sigma_{\widetilde{\omega}}(T_1)$ and $\sigma_\omega(T_2)=\sigma_{\widetilde{\omega}}(T_2)$.
This is for example achieved by a root of unity of prime order that is close to $\omega$. Let 
\[
T_2=K_{u(T_2)}\to K_{u(T_2)-1}\to\cdots\to T_1\to\cdots\to K_1\to K_0=O
\]
be a minimal unknotting sequence for $T_2$ that contains $T_1$.
Since it involves only positive-to-negative crossing changes, we have
\[
\sigma_{{\omega}}(T_2)\geq\sigma_{{\omega}}(K_{u(T_2)-1})\geq\cdots \geq \sigma_{{\omega}}(T_1)\geq\cdots\geq\sigma_{{\omega}}(O)=0
\]
by Lemma~\ref{lemma:Sigundercrossingchange}. Therefore, $\sigma_\omega(T_1)
\leq
\sigma_\omega(T_2)$ holds.
\end{proof}
\begin{Remark}
By the above proof, Proposition~\ref{prop:Sigundercrossingchange} remains true for any knot $K$ with $\frac{s(K)}{2}=u(K)$.
For example, all knots that are closures of positive braids, which include algebraic knots. Instead of half the Rasmussen invariant, we could have used any other knot invariant that satisfies the conditions of Lem\-ma~\ref{lem_tauinv} and that is equal to the unknotting number
on torus knots, compare~\cite{Livingston_Comp}.
\end{Remark}
We prove Theorem~\ref{thm:T(2,n)<T(3,m)} using Proposition~\ref{prop:Sigundercrossingchange}
and the following combinatorial formula for the Levine-Tristram signatures of torus knots,
see \cite{GambaudoGhys_BraidsSignatures}, for $\sigma_{-1}$ it is originally due to Brieskorn and Hirzebruch
\cite{Brieskorn_DifftopovonSing}\cite{Hirzebruch_SingExoticBourbaki}.
We denote the cardinality of a finite set $S$ by $\sharp S$.
\begin{lemma}\label{lemma_signformula}
 Let $n\geq2$ and $m\geq2$ be coprime positive integers.
Set \[S=\{\frac{k}{n}+\frac{l}{m}\;|\;1\leq k\leq n-1 \text{ and } 1\leq l\leq m-1\}\subset [0,2].\]
Then for $\theta \in [0,1]$ we have
\begin{equation*}\label{eq_signformula}
\sigma_{e^{2\pi i\theta}}(T(n,m))=\sharp \left(S\cap[\theta,\theta+1]\right)-\sharp \left(S\backslash(\theta,\theta+1)\right).
\end{equation*}
\end{lemma}

\begin{proof}[Proof of Theorem~\ref{thm:T(2,n)<T(3,m)}]
Fix $n=2k+1$ and note that $m=\lfloor\frac{3}{2}k+1\rfloor$ is minimal with $n \leq \frac{4}{3}m+\frac{1}{3}$.
By Proposition~\ref{prop:T(2,n)<T(3,m)} we have \[T(2,2k+1)\leq_G T(3,\lfloor\frac{3}{2}k+1\rfloor).\]
Together with \[T(3,\lfloor\frac{3}{2}k+1\rfloor)\leq_GT(3,m)\text{ for all }m \geq \lfloor\frac{3}{2}k+1\rfloor,\]
an easy instance of Theorem~\ref{thm:T(n,m)<T(a,b)ifn<am<b},
we conclude that \[T(2,2k+1)\leq_GT(3,m)\text{ for all }m\geq\lfloor\frac{3}{2}k+1\rfloor.\]

For the other direction, we let $n=2k+1$ be any odd number and write $m=\lceil\frac{3}{2}k-1\rceil$,
which is the largest $m$ that does not satisfy $n \leq \frac{4}{3}m+\frac{1}{3}$.
Thus, we have to show that $T(2,2k+1)\nleq_GT(3,m)$.
For $k\leq4$, calculating unknotting numbers yields
\[T(2,5)\nleq_G T(3,2),T(2,7)\nleq_G T(3,4),\et T(2,9)\nleq_G T(3,5).\]
If $k\geq 5$,  we distinguish two cases.
Either, $k$ equals 1 or 2 modulo 4, or $k$ equals 3 or 4 modulo 4.

For $k = 1+4l,2+4l$ with $l\geq 1$, Murasugi's formula for torus knots of braid index $3$, see~\cite[Proposition 9.1]{Murasugi_OnClosed3Braids} or~\cite[Theorem 5.2]{GLM},
provides
\[\s(T(3,m))=2k-2,\]
which is strictly less than
\[\s(T(2,2k+1))=2k.\]
Thus, Proposition~\ref{prop:Sigundercrossingchange} yields $T(2,2k+1)\nleq_GT(3,m)$.

For $k = 3+4l,4+4l$ with $l\geq 1$, Murasugi's formula gives
\begin{equation}\label{eq:sigma(T(3,m))=2n=sigma(T(2,2n+1))}
\s(T(3,m))=2k=\s(T(2,2k+1)).
\end{equation}
In this case $\s$ does not suffice as obstruction directly,
but we use~\eqref{eq:sigma(T(3,m))=2n=sigma(T(2,2n+1))} to calculate $\so(T(3,m))$ for $\omega$ close to $-1$,
which yields the desired obstruction.
More precisely, set
\begin{align}\label{eq:omega}
\omega=e^{2\pi i\theta}\text{ with}\left\{\begin{array}{l}
                 \theta\in(\frac{1}{2}-\frac{2}{3m},\frac{1}{2}-\frac{1}{3m}) \;\text{ for $m$ even}, \text{ i.e.\  } k=3+4l\\
		\theta\in(\frac{1}{2}-\frac{3}{6m},\frac{1}{2}-\frac{1}{6m}) \;\text{ for $m$ odd}, \text{ i.e.\  } k=4+4l
                \end{array}\right. .
\end{align}
By Lem\-ma~\ref{lemma_signformula} the value of $\so(T(3,m))$ is the same for all these $\omega$.
\begin{Claim}\label{claim_sigma}
For all $m=4+6l,5+6l$ with $l\geq 1$ and $\omega$ as in \eqref{eq:omega},
we have $\so(T(3,m))=\s(T(3,m))-2$.
\end{Claim}
Lemma~\ref{lemma_signformula} shows that the above $\omega$ can be chosen such that \[\s(T(2,2k+1))=\so(T(2,2k+1)).\]
Hence, Claim~\ref{claim_sigma} and~\eqref{eq:sigma(T(3,m))=2n=sigma(T(2,2n+1))} yield
\begin{align*}
\so(T(3,m)) &= \s(T(3,m))-2 < \s(T(3,m))\\
 &= \s(T(2,2k+1)) = \so(T(2,2k+1)).
\end{align*}
Therefore, $T(2,2k+1)\nleq_GT(3,m)$ by Proposition~\ref{prop:Sigundercrossingchange}.
It remains to prove Claim~\ref{claim_sigma}.

For the case when $m$ is even,
Lemma~\ref{lemma_signformula} applied to the knot
\[T=T(3,m)=
T(3,4+6l)
\]
yields that $\s(T)$ is
\[
\sharp \left(S\cap[\tfrac{1}{2}-\tfrac{1}{3m}+\epsilon,\tfrac{3}{2}-\tfrac{1}{3m}+\epsilon]\right)-\sharp \left(S\backslash(\tfrac{1}{2}-\tfrac{1}{3m}+\epsilon,\tfrac{3}{2}-\tfrac{1}{3m}+\epsilon)\right)
\]
and that $\so(T)$ is
\[
\sharp \left(S\cap[\tfrac{1}{2}-\tfrac{1}{3m}-\epsilon,\tfrac{3}{2}-\tfrac{1}{3m}-\epsilon]\right)-\sharp \left(S\backslash(\tfrac{1}{2}-\tfrac{1}{3m}-\epsilon,\tfrac{3}{2}-\tfrac{1}{3m}-\epsilon)\right)
\]
for $\epsilon$ small enough.
Observe that
\begin{align*}
\frac{3}{2}-\frac{1}{3m}&=
\frac{2}{3}+\frac{5m-2}{6m}=\frac{2}{3}+\frac{5(4+6l)-2}{6(4+6l)}\\
&=\frac{2}{3}+\frac{3+5l}{4+6l}=\frac{2}{3}+\frac{3+5l}{m}\in S \\
\et\frac{1}{2}-\frac{1}{3m}&=\cdots 
=\frac{1}{3}+\frac{l}{m}+\frac{1}{3}\frac{1}{m}\notin S.
\end{align*}
This means
\[
\left(S\cap[\tfrac{1}{2}-\tfrac{1}{3m}-\epsilon,\tfrac{3}{2}-\tfrac{1}{3m}-\epsilon]\right) \dot{\cup} \{\tfrac{3}{2}-\tfrac{1}{3m}\}=S \cap [\tfrac{1}{2}-\tfrac{1}{3m}+\epsilon,\tfrac{3}{2}-\tfrac{1}{3m}+\epsilon]
\]
and
\[
S\backslash(\tfrac{1}{2}-\tfrac{1}{3m}-\epsilon,\tfrac{3}{2}-\tfrac{1}{3m}-\epsilon)=
\left(S \backslash (\tfrac{1}{2}-\tfrac{1}{3m}+\epsilon,\tfrac{3}{2}-\tfrac{1}{3m}+\epsilon)\right)\dot{\cup} \{\tfrac{3}{2}-\tfrac{1}{3m}\}.
\]
Therefore, $\s(T)=2+\so(T)$ holds.

If $m$ is odd, we have $m=5+6l$.
Similarly to the even case, we get
\[
\frac{3}{2}-\frac{1}{6m}
=\frac{2}{3}+\frac{4+5l}{m}
\in S, \text{ but }\frac{1}{2}-\frac{1}{6m}\notin S.
\]
The rest of the argument is as in the case when $m$ is even.\end{proof}
It remains to prove Lemma~\ref{lemma:Sigundercrossingchange}.
Let $-K$ denote the mirror image of a knot $K$ (with reversed orientation), and let $K_1\sharp K_2$ denote the connected sum of two knots $K_1$ and~$K_2$.
\begin{lemma} \label{lem_tauinv}
Let $\tau$ be a integer valued knot invariant satisfying
\begin{itemize}
 \item $\tau(K_1\sharp K_2)=\tau(K_1)+\tau(K_2)$ and $\tau(-K_1)=-\tau(K_1)$ for all knots $K_1$ and $K_2$,
 \item $\tau(K)\leq g_s(K)$ for all knots $K$,
 \item there exists a knot $K$ with $\tau(K)=1$ that can be transformed to the unknot $O$ by a positive-to-negative crossing change.
\end{itemize}
Then $\tau$ is a concordance invariant,
$\vert\tau(K)\vert\leq g_s(K)$ for all knots $K$, and
\[
0\leq\tau(K_+)-\tau(K_-)\leq1,
\]
whenever $K_-$ is a knot obtained from $K_+$ by a positive-to-negative crossing change.
\end{lemma}
Lemma~\ref{lem_tauinv} is a variation of the statement in \cite[Corollary 2 and 3]{Livingston_Comp}.
The first two assertions are given in \cite[Corollary 2]{Livingston_Comp}.
The proof of the third assertion given in~\cite{Livingston_Comp} needs to be modified as follows to yield a proof Lemma~\ref{lem_tauinv}.
In the proof of~\cite[Corollary 3]{Livingston_Comp}, replace
the knot $T(2,3)$ by a knot $K$ with $\tau(K)=1$ that can be unknotted by changing one positive crossing to a negative one.
This is necessary since we do not want to assume that $\tau(T(2,3))=1$.

We prove Lemma~\ref{lemma:Sigundercrossingchange} by checking the three conditions of Lemma~\ref{lem_tauinv} for $\tau$ equal to $\frac{s}{2}$ and $\frac{\so}{2}$.
\begin{proof}[Proof of Lemma~\ref{lemma:Sigundercrossingchange}]
Rasmussen proves all conditions of Lem\-ma~\ref{lem_tauinv} for $\frac{s}{2}$ in \cite{rasmussen_sInv}
(note that $\frac{s(T(2,3))}{2}=1$).

For $\omega$-signatures, $\sigma_\omega(K_1\sharp K_2)=\sigma_\omega(K_1)+\sigma_\omega(K_2)$ and $\sigma_\omega(-K_1)=-\sigma_\omega(K_1)$
follow from the fact that $A_1\oplus A_2$ is a Seifert matrix for $K_1\sharp K_2$
and $-A_1$ is a Seifert matrix for $-K_1$
if $A_1$ and $A_2$ are Seifert matrices for $K_1$ and $K_2$, respectively.
If $\omega=-1$, we can choose $K$ to be $T(2,3)$ for the third condition since $\s(T(2,3))=2$.
In general, fix a root of unity $\omega$ of prime order.
For a positive integer $k$, let $T(2k-1)$ be the positive twist knot with $2k-1$ half-twists, see Figure~\ref{fig:twistknot}.
\begin{figure}[h]
\centering
\includegraphics[width=0.5\textwidth]{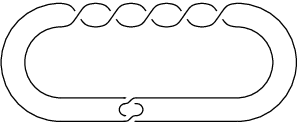}
\caption{The twist knot $T(5)$.}
\label{fig:twistknot}
\end{figure}
These knots can be unknotted by a positive-to-negative crossing change and
\[A=\left[\begin{matrix} k&1\\
0&1\end{matrix}\right]\]
is a Seifert matrix for $T(2k-1)$.
For sufficiently large $k$, both eigenvalues of the Hermitian matrix $(1-\omega) A +(1-\overline{\omega})A^t$ are positive.
Thus, choosing $K$ to be $T(2k-1)$ for a sufficiently large $k$, provides the third condition. 
\end{proof}

\section{A bound on Gordian adjacency for torus knots of higher braid indices}\label{sec:higherIndex&Asymtotics}
This section 
is concerned with the question,
when is $T(a,n)\leq_GT(b,m)$ for fixed $a<b$ and $n,m$ large?
Concretely, we study the numbers
\[
\underline{c}(a,b)=\lim\inf_{m\to\infty}{\frac{n(m)}{m}}\et\overline{c}(a,b)=\lim\sup_{m\to\infty}{\frac{n(m)}{m}},
\]
where $n(m)$ denotes the largest integer such that $T(a,n(m))\leq_GT(b,m)$.
We suspect, but cannot prove, that \[\underline{c}(a,b)=\overline{c}(a,b)\text{ for all }a < b\in\N.\]
Certainly $\underline{c}(2,3)=\overline{c}(2,3)=\frac{4}{3}$ by Theorem~\ref{thm:T(2,n)<T(3,m)}.
Also note that $1\leq \underline{c}(a,b)$ by Theorem~\ref{thm:T(n,m)<T(a,b)ifn<am<b}
and $\overline{c}(a,b)\leq \frac{b-1}{a-1}$ since
\[
\frac{(a-1)(n(m)-1)}{2}=u(T(a,n(m)))\leq u(T(b,m))=\frac{(b-1)(m-1)}{2}.
\]
Using $\omega$-signatures we get an upper bound for $\overline{c}(a,b)$ that is strictly better than $\frac{b-1}{a-1}$.

\begin{prop}\label{prop:upperboundforGadjacency} If $a\leq b\in\N$, then
\[
\overline{c}(a,b)\leq \frac{a\lceil\frac{b}{a}\rceil^2-(a+2b)\lceil\frac{b}{a}\rceil+b(b+1)}{(a-1)b}\leq\frac{b}{a}.
\]
\end{prop}

A calculation shows that $\frac{a\lceil\frac{b}{a}\rceil^2-(a+2b)\lceil\frac{b}{a}\rceil+b(b+1)}{(a-1)b}=\frac{b}{a}$
if and only if $a$ divides $b$. 
If for example $b-a$ equals $1$,
Proposition~\ref{prop:upperboundforGadjacency} yields
\[
\overline{c}(a,a+1)\leq \frac{a+2}{a+1}.
\]
This is better than $\frac{b}{a}=\frac{a+1}{a}$ or even $\frac{b-1}{a-1}$,
but we only know it to be optimal for $a=2$; namely, $\underline{c}(2,3)=\overline{c}(2,3)=\frac{4}{3}$.
Note that Proposition~\ref{prop:upperboundforGadjacency} is strictly better than what one gets using the classical signature.
For example, the signature provides only $\overline{k}(3,4)\leq\frac{3}{2}$ since $\s(T(3,n))\sim \frac{4}{3}n$ and $\s(T(4,m))\sim2m$ by~\cite[Proposition 9.1, Proposition 9.2]{Murasugi_OnClosed3Braids},
which is the same factor one gets when using that the unknotting number has to decrease.

\begin{proof} We use the following approximation by Gambaudo and Ghys~\cite[Proposition 5.2]{GambaudoGhys_BraidsSignatures}.
Let $l$ be a positive integer and $\theta$ a real number with $\frac{l-1}{b}<\theta\leq\frac{l}{b}$, then
\begin{equation}\label{eq:asyptoticomegasigfortoruslinks}
\left|\s_{e^{2\pi i\theta}}(T(b,m))-m\left(2(b-(2l-1))\theta+\frac{2l(l-1)}{b}\right)\right|\leq 2b.
\end{equation}

Proposition~\ref{prop:Sigundercrossingchange} yields $\so(T(b,m))-\so(T(a,n(m)))\geq0$. 
By the approximation we get
\begin{equation}\label{eq:ghys}
\begin{split}
m\left(2(b-(2l-1))\theta+\frac{2l(l-1)}{b}\right)\\ -n(m)\left(2(a-(2l'-1))\theta+\frac{2l'(l'-1)}{a}\right)
\geq -2(a+b),
\end{split}
\end{equation}
where $l$ and $l'$ are positive integers with $\frac{l-1}{b}<\theta\leq\frac{l}{b}$ and $\frac{l'-1}{a}<\theta\leq\frac{l'}{a}$, respectively. 
Choosing $\theta=\frac{1}{a}$ 
inequality \eqref{eq:ghys} becomes
\[
m\left(2\frac{(b-(2\lceil\frac{b}{a}\rceil-1))}{a}+2\frac{\lceil\frac{b}{a}\rceil(\lceil\frac{b}{a}\rceil-1)}{b}\right)-n(m)2\frac{a-1}{a}\geq -2(a+b)
\]
or equivalently
\begin{equation}\label{eq:upperboundforGadjacency}
\frac{n(m)}{m}\leq \frac{a\lceil\frac{b}{a}\rceil^2-(a+2b)\lceil\frac{b}{a}\rceil+b(b+1)}{(a-1)b} + \frac{a(a+b)}{m(a-1)}.
\end{equation}
This proves the first inequality.\footnote{Note the following technical point.
If $\omega=e^{2\pi i \frac{1}{a}}$ is not a root of unity of prime order, 
then Lemma~\ref{lemma:Sigundercrossingchange} cannot be applied as above.
Instead one chooses a sequence of $\theta_{k}$ tending to $\frac{1}{a}$,
such that every $e^{2\pi i \theta_k}$ is a root of unity of prime order. 
}
The second inequality can be checked by a calculation.
\end{proof}

\begin{Remark}\label{rmk:optSigbound}
 Our choice $\theta=\frac{1}{a}$ is the best possible and yields the optimal bound for $\overline{c}(a,b)$
that can be achieved using the properties of signatures from Lemma~\ref{lemma:Sigundercrossingchange}.
This can be checked using the above approximation from \cite{GambaudoGhys_BraidsSignatures}.
\end{Remark}

In order to determine $\underline{c}(a,b)$ and $\overline{c}(a,b)$ for $(a,b)\neq (2,3)$,
we now wish to find geometric constructions in the spirit of Section~\ref{sec:T(2,?)<T(3,?)}
that at least for some $a$ and $b$ yield a lower bound for $\underline{c}(a,b)$ that is equal to the upper bound given by Proposition~\ref{prop:upperboundforGadjacency}.
So far we have only found constructions giving lower bounds that do not coincide with the upper bounds,
e.g.\ $\frac{5}{3}\leq \underline{c}(2,4)\leq\overline{c}(2,4)\leq 2$ and $\frac{9}{8}\leq\underline{c}(3,4)\leq\overline{c}(3,4)\leq \frac{5}{4}$.

\section{Algebraic adjacency}\label{sec:algebraically}
In this section we compare $\leq_G$ with an adjacency notion for plane curve singularities.
We first recall the notion of an algebraic knot following Milnor \cite{milnor_SingularPoints}.
Let $f\colon(\C^2,0)\to(\C,0)$ be a polynomial function or a holomorphic function germ that is
irreducible\footnote{With the weaker assumption `square-free' most of what is done in this section still works, but we get links instead of knots.}
in the ring of holomorphic function germs $\C\{x,y\}$
and has an isolated singularity at the origin.
The transversal intersection of its zero set $V(f)\subseteq \C^2$
with a sufficiently small sphere around the origin $S^3_\varepsilon=\{(x,y)\in\C^2\;\vert\;\Vert x \Vert^2+\Vert y\Vert^2=\varepsilon^2\}$ is a knot in $S^3_\varepsilon\cong S^3$
called the \emph{knot of the singularity of $f$}.
For example, the torus knot $T(n,m)$ is the knot of the singularity of $x^n-y^m$.
In this case the small sphere can be taken to be the standard unit sphere $S^3$;
thus, $T(n,m)=S^3\cap\lbrace (x,y)\in \C^2\;\;|\;\;x^n-y^m=0\rbrace\subset S^3$.
Knots that can occur as knots of singularities are called \emph{algebraic}.

Arnol'd studied adjacency of singular function germs \cite[Definition 2.1]{Arnold_normalforms}, see also \cite{siersma_diss}.
As we are interested in knots, we study singular function germs only up to topological type,
i.e.~up to the isotopy class of their knots of singularity, see e.g.~\cite{brieskornknoerrer}.
Thus, we use the following version of adjacency. 

A \emph{deformation} of $f\in\C\{x,y\}$ is a smooth family $h_t\in\C\{x,y\}$, defined for small enough real $t\geq0$, with $h_0=f$.
\begin{Definition}
Let $K_1$ and $K_2$ be algebraic knots. We say $K_1$ is \emph{algebraically adjacent} to $K_2$, denoted by $K_1\leq_a K_2$, if there exists a germ $f\in\C\{x,y\}$ with $K_2$ as knot of the singularity
and a deformation $h_t$ of $f$,
such that for small nonzero $t$ the germ $h_t$ has $K_1$ as knot of the singularity.
\end{Definition}
\begin{Remark}\label{muadj=algadj}
Since every holomorphic germ yields the same knot as its Taylor polynomials of large enough degrees,
one can study polynomials or holomorphic germs.

Isotopy classes of algebraic knots can be identified canonically with $\mu$-constant-homotopy classes of irreducible germs $(\C^2,0)\to (\C,0)$,
where $\mu$ is the Milnor number.
With this identification the above notion of adjacency for algebraic knots corresponds to
the concept of $\mu$-adjacency studied by Siersma in \cite{siersma_diss}. 
We sketch this identification.
If two plane curves can be connected by a $\mu$-constant path, then the associated algebraic knots are isotopic \cite{Le_Ramanujam_73_InvMilnorNRimpliesInvTopType}.
For the converse, assume that two irreducible 
germs $f_0$ and $f_1$ have the same knot of singularity.
After coordinate changes they are both of the form $y^m+c_{m-1}(x)y^{m-1}+\cdots+c_0(x)$,
where $m$ is the multiplicity of $f_0$ and $f_1$, and where the $c_k\in\C\{x\}$ are holomorphic germs with $c_k(0)=0$.
If $f_0$ and $f_1$ have the same knot of singularity,
then they have the same essential terms in their corresponding Puiseux expansions $y_0(x^\frac{1}{m})$ and $y_1(x^\frac{1}{m})$, see e.g.\ \cite{brieskornknoerrer}.
Thus, the two Puiseux expansions can be connected by a family of Puiseux expansions $y_t(x^\frac{1}{m})$ with the same essential terms. This yields a $\mu$-constant family of germs
\[
f_t=\prod_{\xi^m=1}\left(y-y_t(\xi x^\frac{1}{m})\right)\in \C\{x,y\}
\]
that connects $f_0$ to $f_1$.
\end{Remark}

As described in the introduction, algebraic adjacency has similar
properties as Gordian adjacency.
For example, Theorem~\ref{thm:T(n,m)<T(a,b)ifn<am<b} is known and easy to show for $\leq_a$ instead of $\leq_G$.

\begin{prop}\label{prop:T(n,m)<T(a,b)ifn<am<b(algebraic)}
If $n\leq a$ and $m\leq b$, then T$(n,m)\leq_aT(a,b)$.
\end{prop}

\begin{proof}
 The torus knot $T(a,b)$ is the knot of the singularity $y^a-x^{b}$. We choose as deformation $h_t(x,y)=y^a-x^{b}+t(y^n-x^{m})$.
For $t$ small (but fixed), we perform, in a small chart around the origin, a biholomorphic coordinate change,
which does not change the topological type of the singularity,
such that $h_t=y^n(t+y^{a-n})-x^m(t+x^{b-m})$ becomes $y^n - x^m$. To be explicit, the coordinate change is given as the inverse of the holomorphic map $(x,y)\mapsto (x\sqrt[m]{t+x^{b-m}},y\sqrt[n]{t+y^{a-n}})$.

\end{proof}

The obstruction to Gordian adjacency given in Corollary~\ref{cor:K1<K2->g(K2)-g(K1)>sigmdiff/2)} also holds for algebraic adjacency.
Actually, Corollary~\ref{cor:K1<K2->g(K2)-g(K1)>sigmdiff/2)} and its counterpart for $\leq_a$ are a consequence of the fact
that $\vert\frac{\sigma_\omega(K_2)-\sigma_\omega(K_1)}{2}\vert$ is less than
or equal to the cobordism distance of $K_1$
and $K_2$, and the following.
For algebraic knots, both $K_1\leq_G K_2$ and $K_1\leq_a K_2$
yield a cobordism in $S^3\times[0,1]$ between $K_1$ and $K_2$ of minimal genus $u(K_2)-u(K_1)=g_s(K_2)-g_s(K_1)$.
For an algebraic adjacency given by a deformation $h_t$, this cobordism is given as follows.
Let $S_2$ be a sufficiently small sphere with $K_2=S_2\cap V(h_0)$.
Then, by transversality, $t$ can be chosen small enough such that $S_2\cap V(h_t)$ is still $K_2$ and $K_1=S_1\cap V(h_t)$ for a small enough sphere $S_1$.
By a small perturbation of $h_t$, the zero-set $V(h_t)$ becomes a smooth algebraic curve $F$
with $K_2=S_2\cap F$ and $K_1=S_1\cap F$.
The cobordism between $K_1$ and $K_2$, which is given by $F$,
has minimal genus $u(K_2)-u(K_1)=g_s(K_2)-g_s(K_1)$ by the Thom conjecture~\cite[Corollary 1.3]{KronheimerMrowka_Gaugetheoryforemb}.

Despite similarities, Gordian adjacency and algebraic adjacency do not agree for algebraic knots or even torus knots.
The obstruction given in Proposition~\ref{prop:Sigundercrossingchange}
does not hold for algebraic adjacency.
Concretely, we have $T(2,15)\nleq_GT(3,10)$ by Theorem~\ref{thm:T(2,n)<T(3,m)}, but $T(2,15)\leq_aT(3,10)$,
which we show now.
The next proposition generalizes the algebraic adjacency $T(2,6)\leq_a T(3,4)$ calculated by Arnol'd~\cite[$A_5\leftarrow E_6$]{Arnold_normalforms}.
This gives a large class of examples of algebraic adjacencies of torus links, including $T(2,15)\leq_aT(3,10)$,
which are not covered by Proposition~\ref{prop:T(n,m)<T(a,b)ifn<am<b(algebraic)}.

\begin{prop}\label{prop:a,bc->b,ac}
  Let $a,b,c$ be positive integers with $a\leq b$, then $T(a,bc)\leq_aT(b,ac)$. In particular, $T(2,3c)\leq_aT(3,2c)$.
\end{prop}

\begin{proof}
   Suppose $a<b$, and choose $f_t=y^b-(x^c-ty)^a$ as deformation. Since $T(b,ac)$ is the knot of the singularity $f_0=y^b-x^{ac}$,
it remains to show that for small $t\neq0$, the knot of the singularity $f_t$ is $T(a,bc)$.
We fix an arbitrary $t>0$ and change coordinates locally around the origin by $(x,y)\mapsto (x,\frac{x^c-y}{t})$.
With this  $f_t$ becomes $(\frac{x^c-y}{t})^b-y^a$. For all monomials of $f_t$ except $-y^a$, the bi-degree\textemdash the tuple of integers consisting of the $x$ degree and $y$ degree of a monomial\textemdash
lies on the line in $\Z^2$ that goes through $(bc,0)$ and $(0,b)$.
This shows that  $f_t$ and $x^{bc}-y^a$ have the same two monomials with bi-degree on the line $(bc,0)$ and $(0,a)$ and all other bi-degrees lie strictly above this line. Therefore,
they have the same knot of singularity by a result of Kouchnirenko~\cite[Corollaire 1.22]{Kouchnirenko_NewtonPoly&MilnorNr}.
\end{proof}

\begin{Remark}
Proposition~\ref{prop:a,bc->b,ac} gives an algebraic proof of an observation by Baader, which states that the cobordism distance of $T(a,bc)$ and $T(b,ac)$ is equal to
\[
\frac{bc+a-ac-b}{2}=\frac{(b-a)(c-1)}{2}
\] 
and which is a key proposition in~\cite{Baader_ScissorEq}.

\end{Remark}

Proposition~\ref{prop:a,bc->b,ac} shows that if we define an algebraic counterpart of $c(a,b)$ in Section~\ref{sec:higherIndex&Asymtotics},
it is larger than or equal to $\frac{b}{a}$, whereas in the Gordian setting $c(a,b)$ is less than or equal to $\frac{b}{a}$ by Proposition~\ref{prop:upperboundforGadjacency}.
Thus, asymptotically, whenever $T(a,n)\leq_GT(b,m)$ for $a\leq b$, we get roughly $n\leq\frac{b}{a}m$ and, therefore, $T(a,n)\leq_aT(b,m)$.
We take this as evidence to conjecture that Gordian adjacency implies algebraic adjacency for torus knots.

\bibliographystyle{alpha}
\bibliography{peterbib}

\end{document}